\RequirePackage{fix-cm}
\documentclass{article}

\usepackage{amssymb}
\usepackage{mathrsfs}
\usepackage{graphicx}
\usepackage{graphicx}
\usepackage{amsmath}
\usepackage{caption}
\usepackage{flushend,cuted}
\usepackage{color}
\usepackage{float}
\usepackage{subfigure}
\usepackage{multirow}

\newcommand{\be}{\begin{equation}}
\newcommand{\ee}{\end{equation}}
\newcommand{\bit}{\begin{itemize}}
	\newcommand{\eit}{\end{itemize}}

\def\diag{{\rm diag}}
\def\Diag{{\rm Diag}}

\renewcommand{\Re}{{\rm I}\!  {\rm R}}

\def\rank{{\rm rank}}
\newtheorem{proposition}{Proposition}[section]
\newtheorem{theorem}[proposition]{Theorem}
\newtheorem{lemma}[proposition]{Lemma}

\newtheorem{definition}[proposition]{Definition}
\newtheorem{remark}[proposition]{Remark}

\newtheorem{algorithm}[proposition]{Algorithm}

\newtheorem{proof}[proposition]{proof}

\renewcommand{\Re}{{\rm I}\!  {\rm R}}

\begin{document}
\title{Feasibility and A Fast Algorithm for Euclidean Distance Matrix Optimization with Ordinal Constraints}
\author{Sitong Lu\thanks{School of Mathematics and Statistics, Beijing Institute of Technology, Beijing, 100081, P. R. China}
	\and Miao Zhang \thanks{School of Mathematics and Statistics, Beijing Institute of Technology, Beijing, 100081, P. R. China}
	\and Qingna Li\thanks{Corresponding author. This author's research is supported by NSFC 11671036. School of Mathematics and Statistics/Beijing Key Laboratory on MCAACI, Beijing Institute of Technology, Beijing, 100081, P. R. China. Email: qnl@bit.edu.cn}}

\maketitle

\begin{abstract}
Euclidean distance matrix optimization with ordinal constraints (EDMOC) has found important applications in sensor network localization and molecular conformation. It can also be viewed as a matrix formulation of multidimensional scaling, which is to embed $n$ points in a $r$-dimensional space such that the resulting distances follow the ordinal constraints. The ordinal constraints, though proved to be quite useful, may result in only zero solution when too many are added, leaving the feasibility of EDMOC as a question. In this paper, we first study the feasibility of EDMOC systematically. We show that if $r\ge n-2$, EDMOC always admits a nontrivial solution. 
Otherwise, it may have only zero solution. The latter interprets the numerical observations of 'crowding phenomenon'. Next we overcome two obstacles in designing fast algorithms for EDMOC, i.e., the low-rankness and the potential huge number of ordinal constraints. We apply the technique developed in \cite{Zhou2017A} to take the low rank constraint as the conditional positive semidefinite cone with rank cut. This leads to a majorization penalty approach. The ordinal constraints are left to the subproblem, which is exactly the weighted isotonic regression, and can be solved by the enhanced implementation of Pool Adjacent Violators Algorithm (PAVA). Extensive numerical results demonstrate {the} superior performance of the proposed approach over some state-of-the-art solvers.
\end{abstract}


\section{Introduction}
Euclidean Distance Matrix (EDM) has its deep root in linear algebra \cite{YoungHouseholder,Gower,Schoenberg1935Remarks}. Optimization models based on EDM are widely used in sensor network localization (SNL), molecular conformaiton (MC), multidimensional scaling (MDS) and so on \cite{QiXiuYuan2013,DISCO,Cox}. We refer to \cite{Liberti,Cox,Borg,Dattorro} for the review on EDM and its close relationship with distance geometry, MDS and various applications.

Let $\mathcal S^n$ denote the space of all $n\times n$ symmetric matrices, endowed with the standard inner product. An EDM $D\in\mathcal S^n$ is a matrix whose elements are the squared distances of points $x_1,\dots, x_n\in\Re^r$, i.e., $D_{ij}=\|x_i-x_j\|^2$. Here $r$ is the embedding dimension. EDM optimization is thus to look for an EDM generated by a set of points $\{x_1,\dots,x_n\}\subset\Re^r$ such that the loss function $f(D)$ is minimized. To put it in a general form, we have
\be\label{prob}
\begin{array}{ll}
\min_{D\in\mathcal S^n}& f(D)\\
\hbox{s.t.} & D\hbox{ is an EDM},
\\
& \rank(JDJ)\le r,\\
& D\in\mathcal P.
\end{array}
\ee
Here the rank constraint guarantees that the embedding dimension is no less than $r$. $J=I-\frac1n ee^T$ is the centralization matrix with the identity matrix $I$ and $e=(1,\dots, 1)^T\in\Re^n$. $\mathcal P$ describes extra constraints on $D$, for example, the box constraints
\[
\mathcal P = \mathcal P_B:=\{D\in\mathcal S^n \ | \ L\le D\le U\}
\]
arising from MC \cite{Fang2013Using}, and the ordinal constraints
\[
\mathcal P = \mathcal P_O:=\{D \in\mathcal S^n\ |\ D_{ij}\ge D_{sk},\ (i,j,s,k)\in\mathcal C\}
\]
arising from nonmetrical multidimensional scaling (NMDS) \cite{Cox,Borg,LiQi2017} where $\mathcal C$ is the set of indices for ordinal constraints. In this paper, we are going to study the EDM optimization with ordinal constraints (EDMOC)
 \be\label{prob-0}
\begin{array}{ll}
\min_{D\in\mathcal S^n}& f(D)\\
\hbox{s.t.} & D\hbox{ is an EDM},
\\
& \rank(JDJ)\le r,\\
& D\in\mathcal P_O.
\end{array}
\ee
Specifically, we will investigate the feasibility of EDMOC and propose a fast algorithm for EDMOC with least squares loss function. Below we give a brief review on the research that motivates our work, followed by our contributions and the organization of the paper. We refer to \cite{Cox,Borg,Kruskal1964Multidimensional,Kruskal1964Nonmetric,deLeeuw1977,deLeeuw2009,SMACOF} for other excellent and popular solvers for vector models of (\ref{prob}) including the famous Scaling by MAjorizing a COmplicated Function (\texttt{SMACOF}).

We start with two equivalent ways of characterising an EDM \cite{Schoenberg1935Remarks,YoungHouseholder}, which are
\be\label{EDM-1}
\diag(D)=0, \ \ -JDJ\succeq 0
\ee
and \be
\label{EDM-2}
\diag(D)=0, \ \ -D\in \mathcal{K}^n_+.
\ee
 Here $\diag(D)$ is the vector formed by the diagonal elements of $D$, and $A\succeq 0$ means that $A\in\mathcal S^n$ is a positive semidefinite matrix. $\mathcal{K}_+^n$ is a conditional positive semidefinite cone defined by
\be
\mathcal{K}_+^n=\{D\in \mathcal S^n \ | \ v^TDv\ge0, \ \forall\ v\in \Re^n,\ v^Te=0\}.
\ee
Based on the characterization (\ref{EDM-1}), there is a large body of publications dealing with EDMOC by semidefinite programming (SDP), which is out of the scope of our paper. We refer to \cite{biswas2004semidefinite,DISCO,Fang2013Using,Toh2008}  just to name a few of outstanding SDP based approaches for EDM optimization in SNL and MC. The characterization (\ref{EDM-2}) has fundamental differences from (\ref{EDM-1}) \cite{Qi2013} as it describes an EDM via the conditional positive semidefinite cone, based on which great progress has been made on numerical algorithms for EDM optimization \cite{Qi2013,Qi2014Computing,QiXiuYuan2013,Qi2014,QiDing2015,LiQi2017,Zhou2017A,Qi2018}, as we will detail below.

In \cite{Qi2013}, a semismooth Newton's method was proposed to solve (\ref{prob}) with $r=n$, $f(D)=f^{LS}(D):= \frac12\|D-\Delta\|_F^2 $ and omitting extra constraints $\mathcal P$, i.e., the nearest EDM problem 
\be\label{prob-edm-1}
\begin{array}{ll}
	\min_{D\in\mathcal S^n} & \frac12\|D-\Delta\|_F^2\\
	\hbox{s.t.} & \diag(D)=0, \ -D\in \mathcal{K}^n_+.
\end{array}
\ee
 Here $\Delta = (\delta_{ij})$ was given. The characterization (\ref{EDM-2}) was used, which was the key to the success of semismooth Newton's method for solving the dual problem of (\ref{prob-edm-1}).
A majorized penalty approach \cite{Qi2014} was further proposed to deal with the low dimensional embedding, i.e.,
\be\label{prob-edm-2}
\begin{array}{ll}
	\min_{D\in\mathcal S^n} & \frac12\|D-\Delta\|_F^2\\
	\hbox{s.t.} & \diag(D)=0, \ -D\in \mathcal{K}^n_+,\\
	& \hbox{rank}(JDJ)\le r,
\end{array}
\ee
where $r$ was the prescribed embedding dimension. A penalty function was used to tackle the rank constraint. Note that
 full spectral decomposition was required in order to compute the majorization function of $\hbox{rank}(JDJ)\le r$.

Inspired by \cite{Qi2013,Qi2014Computing}, Li and Qi \cite{LiQi2017} proposed an inexact smoothing Newton method for EDMOC (\ref{prob-0}) with $f=f^{LS}$ and $r=n$. That is, 
\be\label{prob-edm-3}
\begin{array}{ll}
	\min_{D\in\mathcal S^n} & \frac12\|D-\Delta\|_F^2\\
	\hbox{s.t.} & \diag(D)=0, \ -D\in \mathcal{K}^n_+,\\
	&D\in\mathcal P_O. \\
\end{array}
\ee
As pointed out in \cite{LiQi2017}, the ordinal constraints could improve the quality of embedding points. It naturally happens when the ranking of distances is available, which is exactly the situation in EDMOC.

For box constraints, Zhou et al. \cite{Zhou2017A} recently proposed a majorization-minimization approach to solve (\ref{prob}) with $f$ being the Kruskal's minimization function and $\mathcal P = \mathcal P_B$, i.e.,
\be\label{prob-qi}
\begin{array}{ll}
\min_{D\in\mathcal S^n} & \sum_{i,j} W_{ij}(\sqrt{D_{ij}}-\delta_{ij})^2\\
\hbox{s.t.} & D\in\mathcal P_B, \ -D\in \mathcal{K}_+^n(r),
\end{array}
\ee
where $\mathcal{K}_+^n(r)$ is the conditional semidefinite positive cone with rank cut, defined by
\be
\mathcal{K}^n_+(r)=\{D\in\mathcal S^n \ | \ D\in \mathcal{K}^n_+,\ \rank(JDJ)\le r\}.
\ee
Note that different from the approach proposed in \cite{Qi2014}, the rank constraint is represented by the rank cut of conditional positive semidefinite cone, based on which the following equivalent reformulation is proposed for $-D\in \mathcal{K}_+^n(r)$,
\be\label{g-rank}
-D\in \mathcal{K}_+^n(r)\Longleftrightarrow g(D):={\frac12\|D+\Pi_{K^n_+(r)}(-D)\|^2}=0,
\ee
where $\Pi_{\mathcal{K}^n_+(r)}(\cdot)$ denotes a projection onto $\mathcal{K}_+^n(r)$ (See Section \ref{sec3-1} for details.) A majorization function is proposed for $g(D)$, which allows low computational complexity. Based on such technique, the resulting majorization-minimization approach demonstrates superior numerical performance on MC and SNL. Similar technique is used in \cite{Qi2018} where a robust loss function \[f(D) = f^{RS}(D):= \sum_{i<j} w_{ij}|\sqrt{D_{ij}}-\delta_{ij}|\] is considered. {Zhai and Li \cite{ZhaiLi2019} proposed an Accelerating Block Coordinate Descent method (ABCD) for solving (\ref{prob-qi}) with $f = f^{LS}$.}

Coming back to ordinal constraints, as pointed out in \cite{LiQi2017}, a great number of ordinal constraints may lead to only zero feasible solution, which is numerically observed as 'crowding phenomenon'. A simple example is to take $n = 4$, $r = 1$ in (\ref{prob-0}). Consider the feasible solution of the following set
\[
\{D\in\mathcal S^4\ | \ \diag(D)=0, \ -D\in \mathcal{K}_+^4(1)\}\bigcap \mathcal P_O
\] with
\[
\mathcal P_O=\{D\in\mathcal S^4\ | \ D_{23} \ge D_{12} \ge D_{13} \ge D_{14} \ge D_{34} \ge D_{24} \}.
\]
 We can see that a feasible EDM $D$ satisfying the first four ordinal constraints
\[D_{23} \ge D_{12} \ge D_{13} \ge D_{14} \ge D_{34}\]
can be generated by the corresponding points $x_1,\dots, x_4\in\Re$ as shown in Fig. \ref{fig:counterexample}. However, by adding the last ordinal constraint $ D_{34} \ge D_{24}$, all points collapse to one point
 in order to satisfy the ordinal constraints. In other words, there is no feasible EDM except the zero matrix.
 A natural question thus arises: under what condition does EDMOC admit a nonzero feasible solution? On the other hand, ordinal constraints, as well as the rank constraint, also bring challenges in algorithm design. Consequently, a fast numerical algorithm for EDMOC is still highly in need. It is these observations that motivate the work in this paper. Our main contributions are as follows.
 \begin{figure}[htbp]
	\centering
	\includegraphics[width=4.5cm,height=1cm]{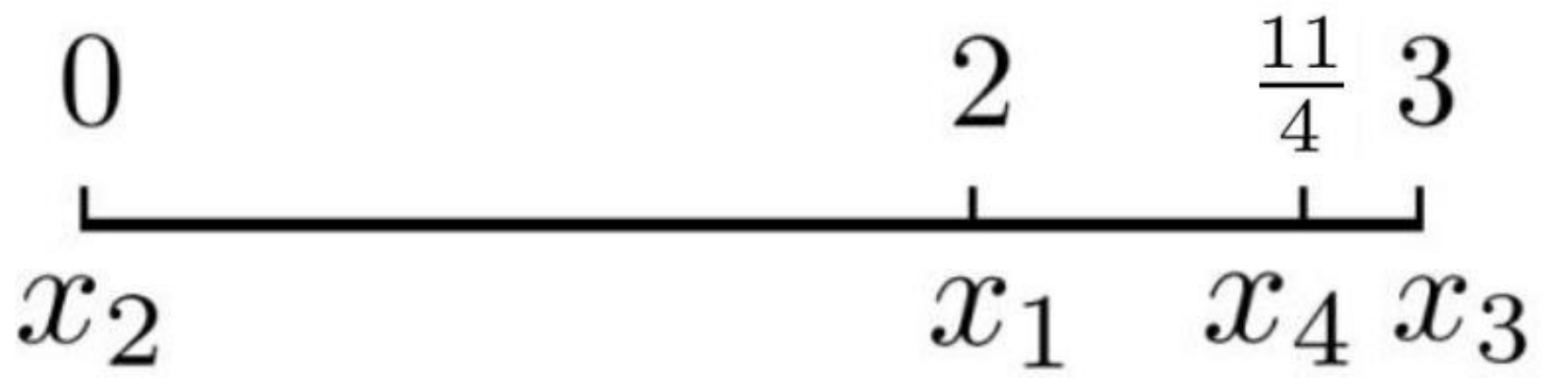}
	\caption{Points generating an EDM $D$ satisfying $D_{23} \ge D_{12} \ge D_{13} \ge D_{14} \ge D_{34}$ in $\Re$.}
	\label{fig:counterexample}	
\end{figure}

 {\bf Our Contributions.} In this paper, we study the equivalent form of EDMOC (\ref{prob-0}), which is
 \be\label{prob-1}
 \begin{array}{ll}
 \min_{D\in\mathcal S^n} & f(D)\\
 \hbox{s.t.} & \diag(D)=0, \ -D\in \mathcal{K}_+^n(r),\\
 & D\in \mathcal P_O.
 \end{array}
 \ee
We study the case where $\mathcal P_O$ describes the full ordinal constraints defined by
 \be\label{Po}
P_O=\{D\in\mathcal S^n \ | \ D_{i_1j_1}\ge D_{i_2j_2}\ge\cdots\ge D_{i_mj_m}\}
 \ee
where $m = \frac{n(n-1)}{2}$, $(i_1,j_1),\ \dots,\ (i_m,j_m)$ are distinct indices of off-diagonal elements in $D$ and $i_1<j_1,\dots, i_m<j_m$.

 The first contribution is that we systematically study the feasibility of (\ref{prob-1}). The main results are as follows.
 \bit
 \item Without rank constraint, (\ref{prob-1}) always admits a nonzero feasible solution.
 \item {$r\ge n-2$}: (\ref{prob-1}) always admits a nontrivial feasible solution (See Section \ref{sec2-1} for the definition).
 \item $r<n-2$: (\ref{prob-1}) may only have zero feasible solution, as shown above by the example of $n = 4,\ r = 1$.
 \eit

 Our second contribution is to develop a fast algorithm for solving (\ref{prob-1}). We tackle the rank constraint by using the technique (\ref{g-rank}) proposed in \cite{Zhou2017A}. The ordinal constraints are left to the subproblem which is exactly the weighted isotonic regression. One advantage of the resulting majorization penalty approach is that the majorization function based on $g(D)$ allows low computational complexity which can speed up the solver. Another advantage is that the huge number of ordinal constraints are tackled within {weighted} isotonic regression, which {can be solved by} an enhanced implementation of PAVA \cite{Kruskal1964Nonmetric,barlow1972statistical}.
 Our extensive numerical results on SNL and MC verify the great performance of the proposed algorithm compared with some state-of-the-art solvers.

The organization of the paper is as follows. In Section \ref{sec-2}, we discuss the feasibility of EDMOC (\ref{prob-1}). In Section \ref{sec-3}, we propose the majorized penalty approach for (\ref{prob-1}). In Section \ref{sec-5}, extensive numerical tests are conducted to verify the efficiency of the proposed algorithm. Final conclusions are given in Section \ref{sec-6}.

{\bf Notations.} We use $|A|$ to denote the number of elements in a set $A$. We use $\|\cdot\|$ to denote the Frobenius norm for matrices and $l_2$ norm for vectors.  Let $\Diag(x)$ be the diagonal matrix with diagonal elements {coming} from vector $x$ {and '$\circ$' be the Hadamard product.}

\section{Feasibility of EDMOC}\label{sec-2}
In this section, we will discuss the feasibility issue of EDMOC systematically. We start with a formal statement of the feasibility problem, then some preliminary properties and the main results for feasibility.
\subsection{Statement of Feasibility}\label{sec2-1}
To better define the feasibility of EDMOC (\ref{prob-1}), we introduce the following notes to represent the full ordinal constraints in (\ref{Po}). Let
\[
\bar\pi(n)=\{(1,2), (1,3), \dots, (1,n), (2,3),(2,4),\dots, (2,n), \dots, (n-1,n)\}.
\]
The collections of all permutations of $\bar\pi(n)$ is denoted by $\Pi(n)$, i.e.,
\[
\Pi(n)=\{\pi(n)\ | \ \pi(n)\hbox{ is a permutation of }\bar \pi(n)\}.
\]
Given
\be\label{pi}
\pi(n)=\{(i_1,j_1), \dots, (i_m,j_m)\},
\ee
it represents the indices of the full ordinal constraints in the following way
\be\label{group}
 D_{i_1j_1}\ge D_{i_2j_2}\ge\cdots\ge D_{i_mj_m}.
\ee
We refer to the full ordinal constraints (\ref{group}) as a group of ordinal constraints given by $\pi(n)$ in (\ref{pi}) (Without causing any chaos, we also refer to $\pi(n)$ as a group of ordinal constraints). The feasible set with respect to ordinal constraints $\pi(n)$ is denoted as
 \be\label{Pon}
\Omega_{\pi(n)}=\{D\in\mathcal S^n \ | \ D_{i_1j_1}\ge D_{i_2j_2}\ge\cdots\ge D_{i_mj_m}\}.
 \ee
 Let $E(r)$ be the set of EDM with embedding dimension not exceeding $r$, i.e.,
 \be
 E(r) = \{D\in\mathcal S^n \ | \ \diag(D)=0,\ -D\in \mathcal{K}_+^n(r)\}.
 \ee
 The feasible set of (\ref{prob-1}) is recast as
 \be
 E(r)\bigcap \Omega_{\pi(n)} =: F(\pi(n),r).
 \ee
 If there is no rank constraint, we denote by
 \[
 F(\pi(n))=E\bigcap \Omega_{\pi(n)} , \hbox{ where }E = \{D\in\mathcal S^n \ | \ \diag(D)=0,\ -D\in \mathcal{K}_+^n\}.\]

 A nontrivial solution of $F(\pi(n), r)$ is defined below.

 \begin{definition} For a nonzero feasible solution
 $D$ of $F(\pi(n), r)$, if
 there exist at least two off-diagonal elements $D_{ij},\ D_{kl}$ such that $D_{ij} \ne D_{kl},$ 
  then $D$ is a nontrivial solution.
\end{definition}

The feasibility of (\ref{prob-1}) can be cast as the following questions:
\bit
\item [Q1:] Given $r$ and $\pi(n)\in\Pi(n)$, does $F(\pi(n), r)$ have a nonzero solution?
\item [Q2:] Given $r$ and $\pi(n)\in\Pi(n)$, does $F(\pi(n), r)$ have a nontrivial solution?
\eit

To explain the difference between nonzero solutions and nontrivial solutions, we need the classical multidimensional scaling (cMDS) to allow us to get a set of points from an EDM $D\in\mathcal S^n$. Firstly, conduct spectral decomposition for $-\frac12JDJ$ as
\be\label{cmds1}
-\frac12JDJ = P_1\Diag(\lambda_1,\dots, \lambda_r)P_1^T,
\ee
where $\lambda_1\ge\dots\ge\lambda_r>0$ are the positive eigenvalues, and $P_1\in\Re^{n\times r}$ consists of the corresponding eigenvectors as columns, $r$ is referred {to} as the embedding dimension. Then the embedding points $x_1, \dots, x_n\in\Re^r$ can be obtained by
\be\label{cmds2}
X:=[x_1,\dots, x_n] = \Diag(\lambda_1^{\frac12},\dots,\lambda_r^{\frac12})P_1^T\in\Re^{r\times n}.
\ee
Due to the fact that $J$ has a zero eigenvalue with eigenvecter $e$, there is $r\le n-1$. In other words, $E(n-1)=E$. Moreover,  it is trivial that
\[
E(r_1)\subseteq E(r_2),  \ \hbox{if } r_1\le r_2.
\]
We refer to \cite{Gower,Schoenberg1935Remarks,Torgerson1952Multidimensional,YoungHouseholder,Borg,Dattorro} for detailed description of cMDS and its generalizations.

 Based on cMDS, we have the following observation, which is crucial in our subsequent analysis.

\begin{proposition}\label{prop-1}
Let $\pi(n)$ take the form as (\ref{pi}). $D$ is a feasible solution of $F(\pi(n),r)$ if and only if there exist $x_1,\dots, x_n\in\Re^r$ such that
\be\label{x-order}
\|x_{i_1}-x_{j_1}\|\ge\|x_{i_2}-x_{j_2}\|\ge\cdots\ge\|x_{i_m}-x_{j_m}\|.
\ee
\end{proposition}
{{\bf Remark.}}
Based on Definition \ref{prop-1}, a nonzero feasible solution of $F(\pi(n),r)$ corresponds to a set of points in $\Re^r$ where at least two points are different from each other. A nontrivial feasible solution of $F(\pi(n),r)$ corresponds to a set of points $\Re^r$ where at least two pairwise distances are different. Obviously, a nontrivial feasible solution must be a nonzero solution, but conversely, it is not necessarily true.

\subsection{Preliminary Properties of Ordinal Constraints}
Before presenting the main results, we need to take a further look at different groups of ordinal constraints. We will {illustrate} in this part that some groups of ordinal constraints are actually equivalent to each other. This is based on the observation that for a set of points, changing the labels will result in some corresponding changes in EDM. We summarize it in the following proposition.
\begin{proposition}\label{prop-2}
 Given $x_1,\dots,x_n \in \Re^{r}$, let $D\in\mathcal S^n$ be the corresponding EDM. Let $P\in\Re^{n\times n}$ be any permutation matrix, {i.e., each row and column  of $P$ has only one element equal to $1$ and  others are $0$.}.
 Suppose $\hat x_1,\dots, \hat x_n\in\Re^r$ are given by {\be\label{eq-X}\widehat{X}=[\hat{x}_1,\hat{x}_2,\dots,\hat{x}_n]=
 XP, \hbox{ where } X=[x_1,x_2,\dots ,x_n].\ee}
Denote the EDM given by $\hat x_1,\dots, \hat x_n$ as $\widehat D$. There is
\[
\widehat{D}=P^TDP.
\]
\end{proposition}
\begin{proof}{Denote \[
P = [e_{t_1},\dots, e_{t_n}],
\]
where $e_k\in\Re^n$ denotes the $k$-th column of identity matrix.  With (\ref{eq-X}), we have $\hat x_i= Xe_{t_i}=x_{t_i}$. Therefore,
\[
\widehat D_{ij} = \|\hat x_i-\hat x_j\|^2 = \|x_{t_i}-x_{t_j}\|^2 = D_{t_it_j}.
\]
Together with
\[
(P^TDP)_{ij} = e_{t_i}^TDe_{t_j} = D_{t_it_j},
\]
we obtained that $\widehat{D}=P^TDP$.
}
\end{proof}
Based on Definition \ref{prop-2}, we define the equivalence between two groups of ordinal constraints as follows.
\begin{definition}
We say a group of ordinal constraints $\pi(n)$ is equivalent to another group of ordinal constraints $\pi'(n)$ $($denoted by $\pi(n)\sim \pi'(n)$$)$ if there exists a permutation matrix $P$ such that
\[
{P^TDP}\in \Omega_{\pi'(n)}, \ \forall\ D\in \Omega_{\pi(n)}.
\]
An equivalent class for some groups of ordinal constraints (denoted by $\mathcal O(n)$), is the collection of all groups of ordinal constraints that are equivalent to each other.
\end{definition}

With the definition of equivalent classes, the collections of all groups of ordinal constraints can be viewed as the union of {all} equivalent classes of ordinal constraints. That is,
\[
\Pi(n)= \bigcup_{i =1}^M \mathcal O^i(n).
\]
where $M $ is the number of equivalent classes of ordinal constraints. Furthermore, for each $\mathcal O^i(n)$, there is $|\mathcal O^i(n)|=n!$, which is the number of different ways {to label} $n$ points. This gives the number of equivalent classes as
\[
M = (\frac{n(n-1)}{2})!/n!.
\]

Below we show a simple example.

{\bf Example 2.1.}
Let $n=3$. There is
\[
\Pi(3)={\{\pi^1(3), \dots, \pi^6(3)\}}
\]where the six groups of ordinal constraints are
\begin{eqnarray*}
\pi^1(3) = \{(1,2),\ (1,3),\ (2,3)\},\ & \pi^2(3)=\{(1,2),\ (2,3),\ (1,3)\},&\\
\pi^3(3) = \{(1,3),\ (1,2),\ (2,3)\},\ & \pi^4(3)=\{(1,3),\ (2,3),\ (1,2)\},&\\
\pi^5(3) = \{(2,3),\ (1,3),\ (1,2)\},\ & \pi^6(3)=\{(2,3),\ (1,2),\ (1,3)\}.&
\end{eqnarray*}
It can be verified that
$\pi^i(3)$ is equivalent to $\pi^1(3)$, $i=2,\dots,6$, with permutation matrix $P^i$ given as follows (See Fig. {\ref{fig:triangle} } for corresponding points which generate a feasible EDM $D^i=(P^i)^TD^1P^i$, $i=2,\dots,6$).
\[
P^2=[e_2,e_1,e_3],\ P^3=[e_1,e_3,e_2],\ P^4=[e_3,e_1,e_2],\]\[ P^5=[e_3,e_2,e_1],\ P^6=[e_2,e_3,e_1].
\]
Consequently, all elements in $\Pi(3)$ are equivalent to each other, i.e., all different types of the ordinal constraints are equivalent for $n=3$. In other words,

\[
\Pi(3)= \mathcal O^1(3),
\]
where $|\mathcal O^1(n)|=6$.

\begin{figure}[htbp]
	\centering
	\subfigure[]{
		\label{triangle231}
		\includegraphics[height=2.0cm,width=2.8cm]{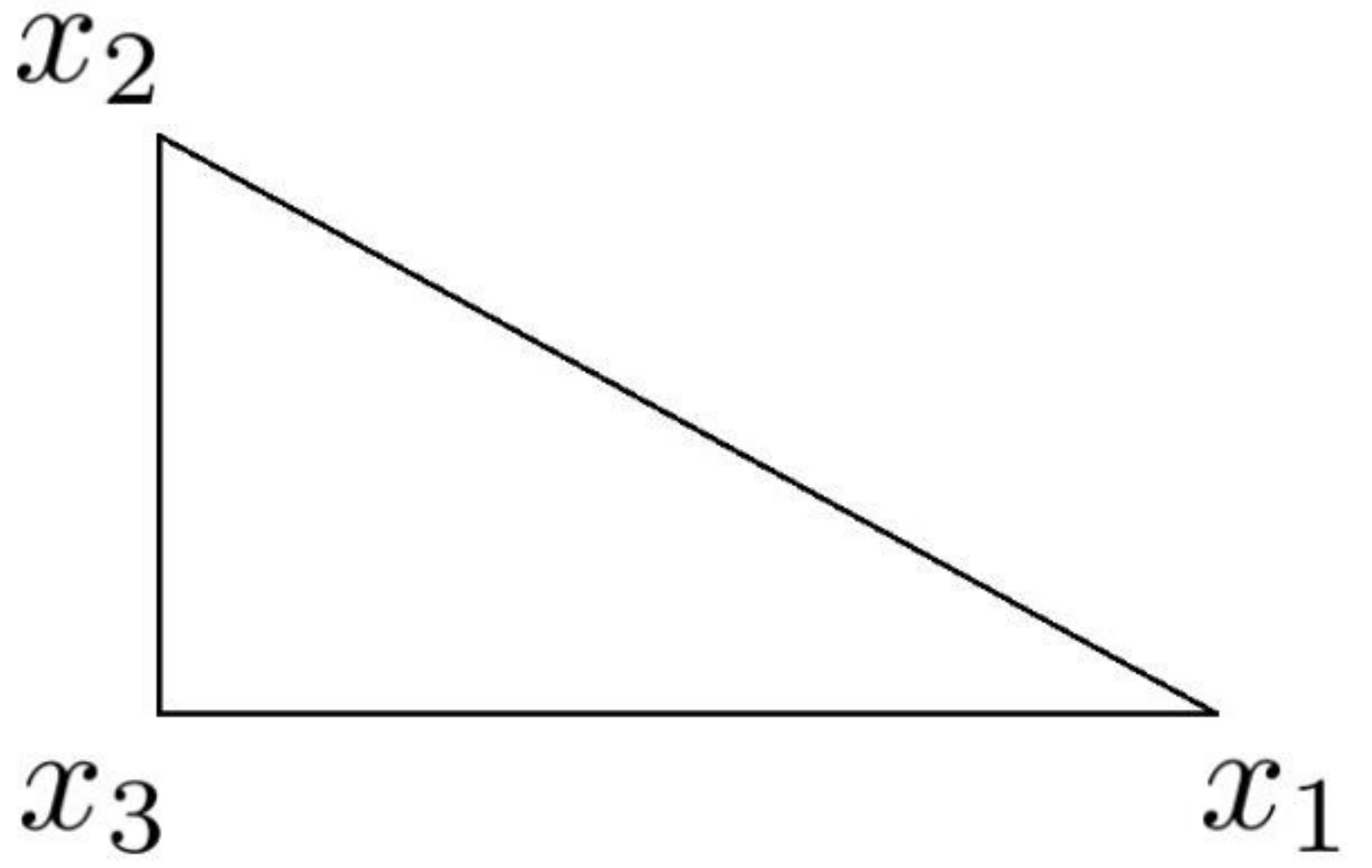}}
	\hspace{0.5in}
	\subfigure[]{
		\label{triangle132}
		\includegraphics[height=2.0cm,width=2.8cm]{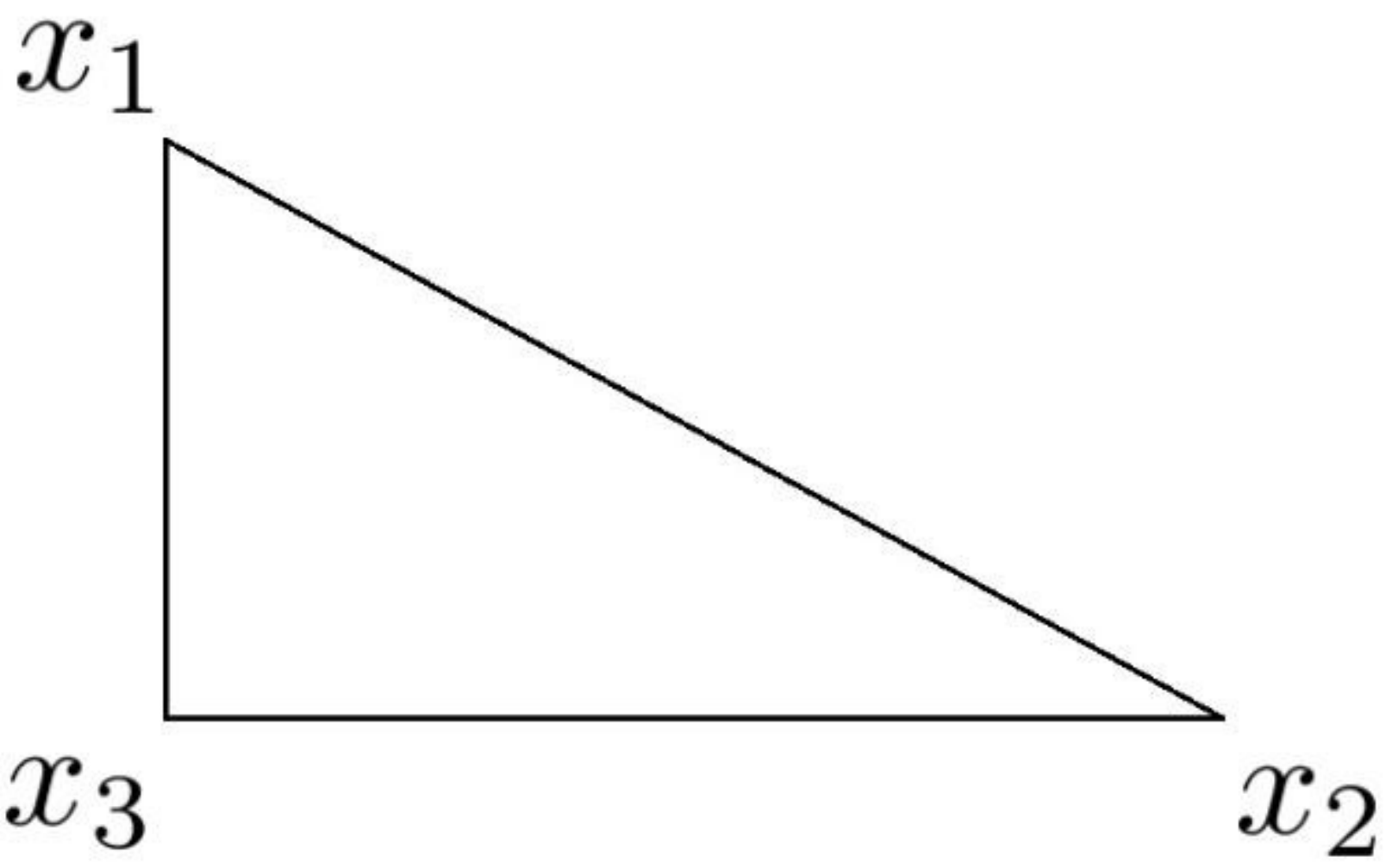}}
	\hspace{0.5in}
	\subfigure[]{
		\label{triangle321}
		\includegraphics[height=2.0cm,width=2.8cm]{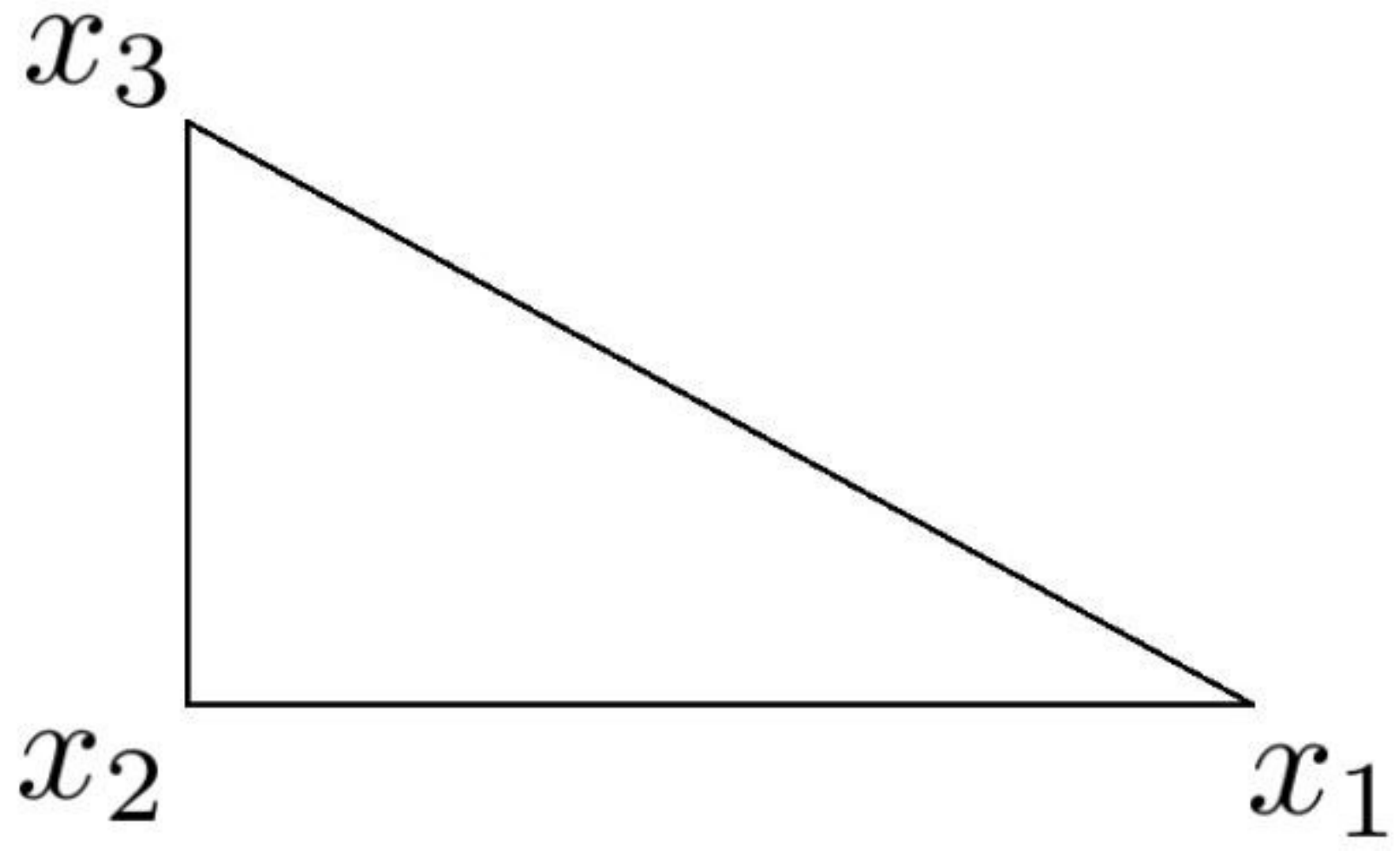}}
	\subfigure[]{
		\label{triangle123}
		\includegraphics[height=2.0cm,width=2.8cm]{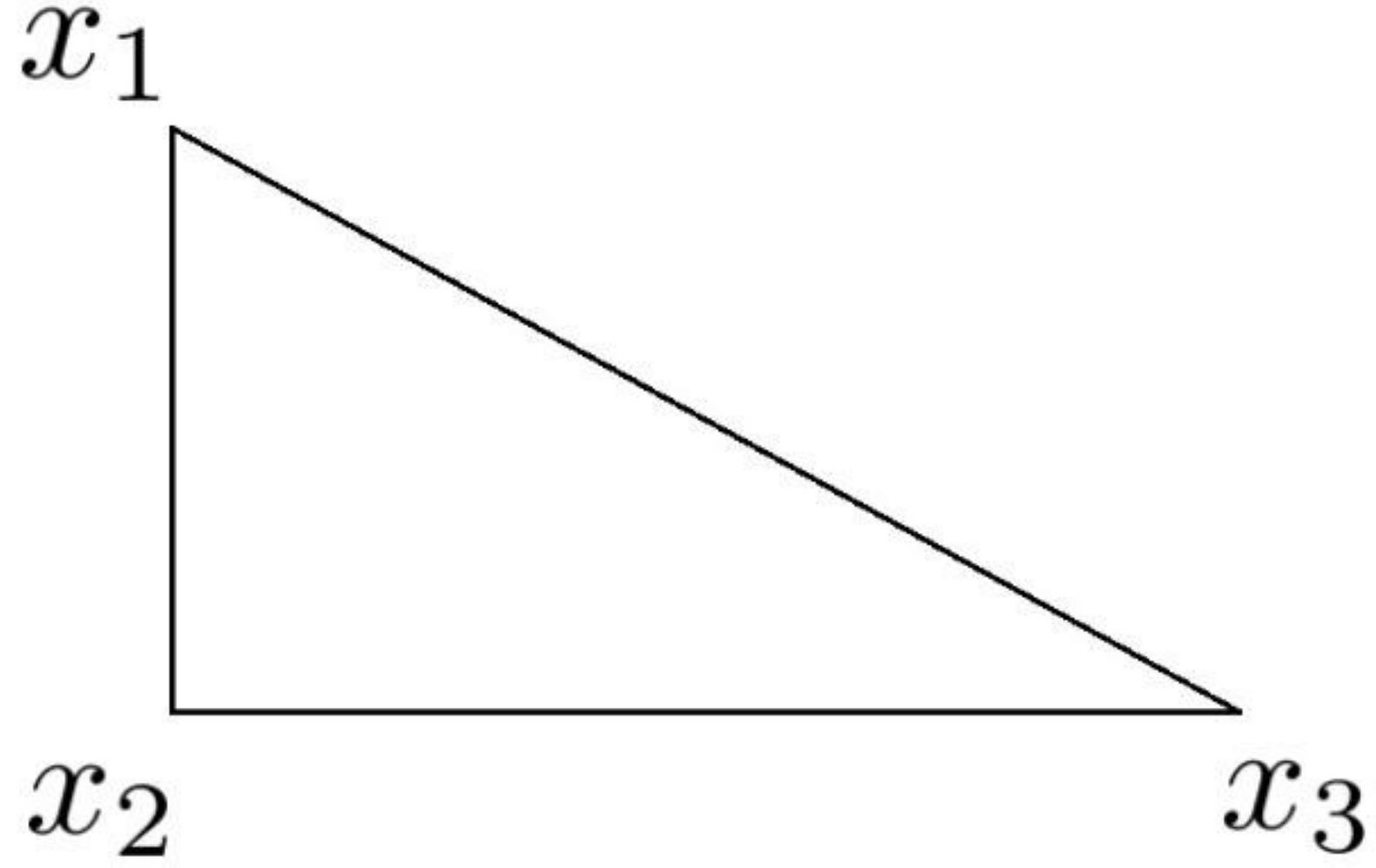}}
	\hspace{0.5in}
	\subfigure[]{
		\label{triangle213}
		\includegraphics[height=2.0cm,width=2.8cm]{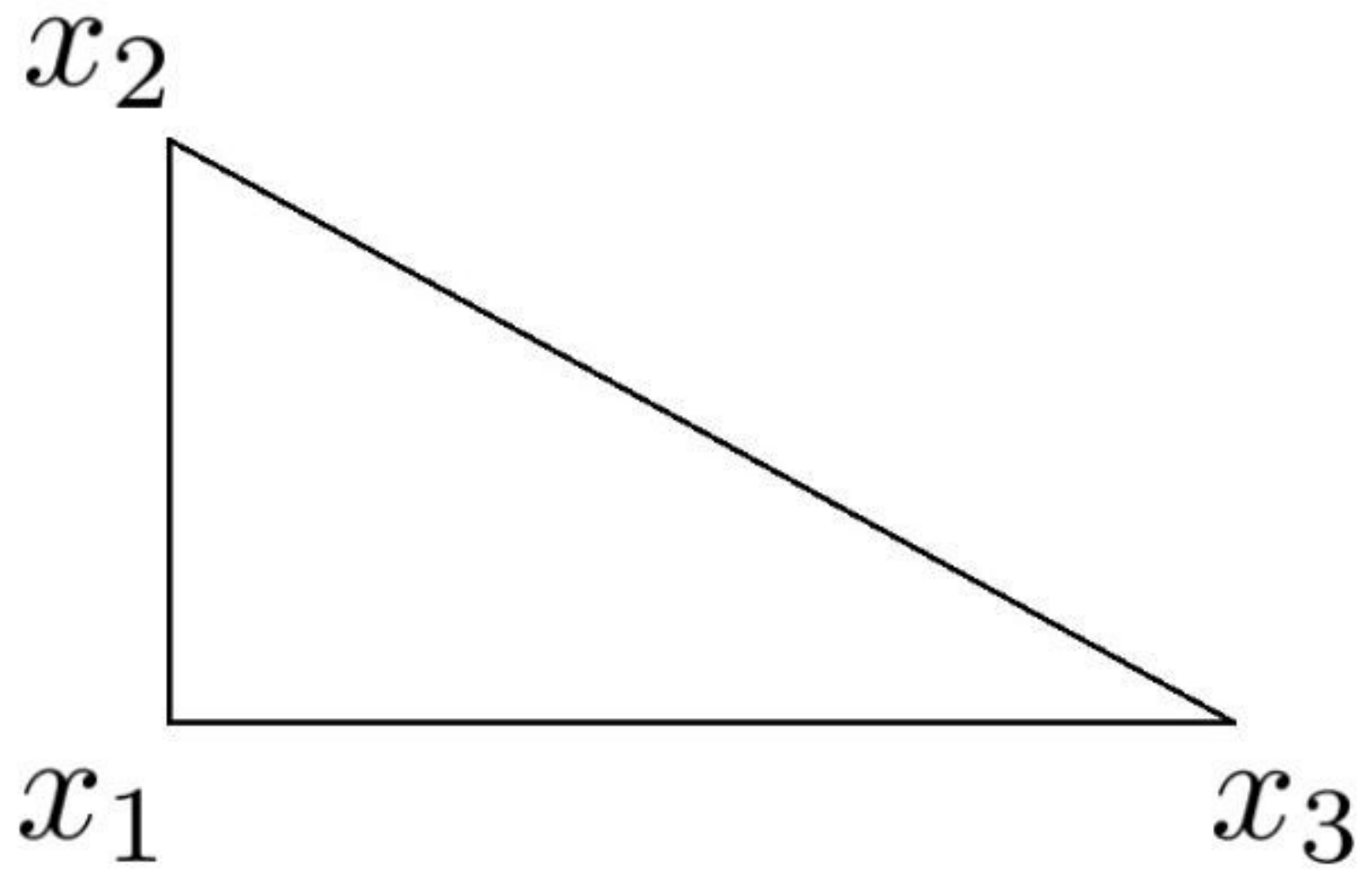}}
	\hspace{0.5in}
	\subfigure[]{
		\label{triangle312}
		\includegraphics[height=2.0cm,width=2.8cm]{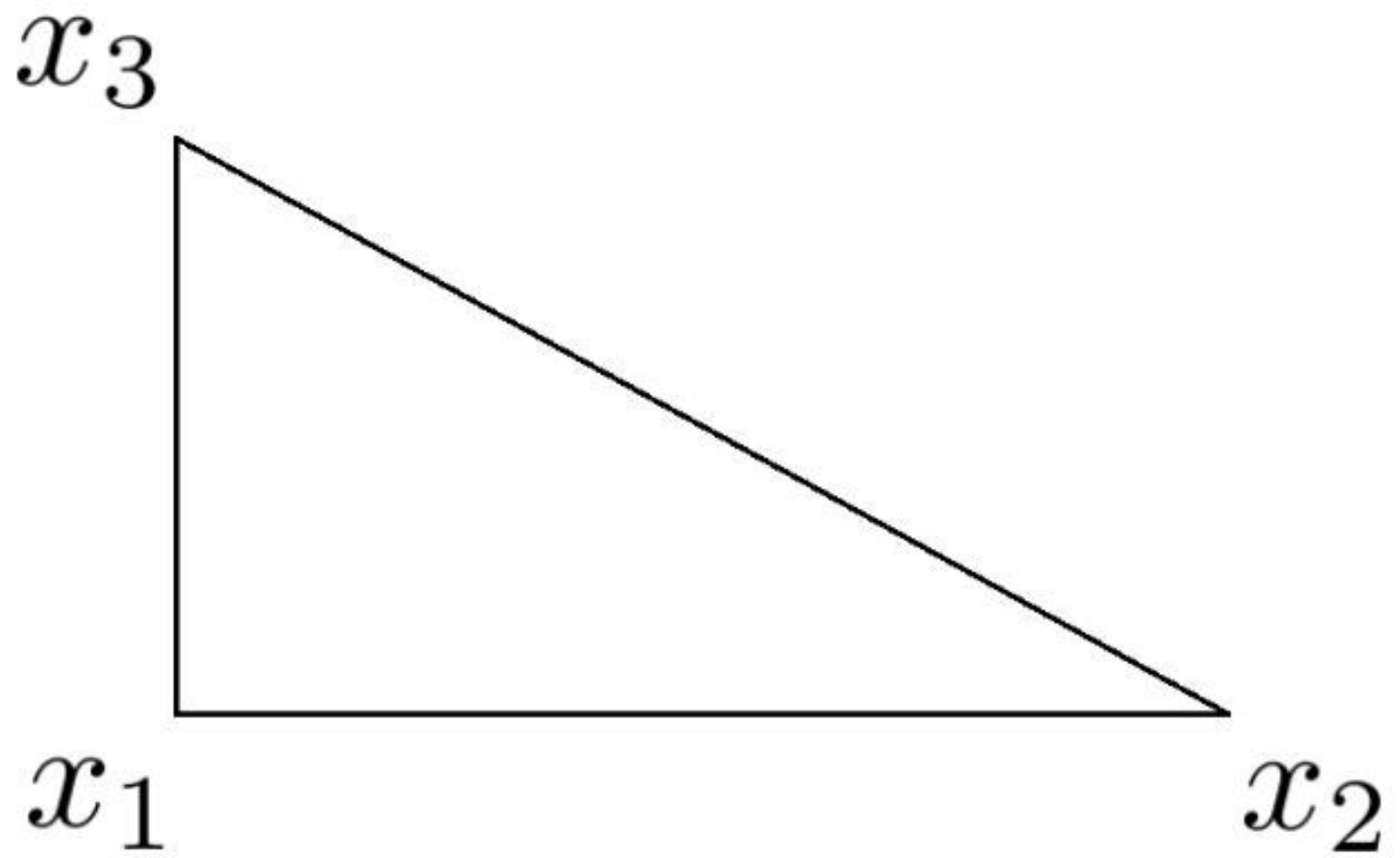}}
	\caption{Points generating an EDM $D^i$, $i=1,\dots,6$ for $(a)$ to $(f)$}	
	\label{fig:triangle}
\end{figure}

The following results are trivial with respect to equivalent classes of ordinal constraints.
\begin{lemma}\label{lem-1}
For two groups of ordinal constraints $\pi^1(n)$ and $\pi^2(n)$ with $\pi^1(n)\sim \pi^2(n)$,
 $F(\pi^1(n),r)$ admits a nonzero solution (or nontrivial solution) if and only if $F(\pi^2(n),r)$ does.

\end{lemma}
	
For Example 2.1, we conclude that
$F(\pi(3), 2)$ admits a nontrivial solution for any {$\pi(3)\in\Pi(3)$.}

\subsection{Main Results}\label{sec2-3}
Now we are ready to give the main results for the feasibility of EDMOC (\ref{prob-1}).
The following theorem partly answers question Q1.
\begin{theorem}[No Rank Constraint]\label{thm-1}
$F(\pi(n))$ admits a nonzero feasible solution for any $\pi(n)\in\Pi(n)$.
\end{theorem}

\begin{proof}
 By \cite{elte1912semiregular}, there exist $x_1,\dots, x_n\in\Re^{n-1}$ such that they form an ($n-1$)-dimensional regular simplex. In other words, there is
 \begin{displaymath}
 	\|x_i-x_j\|=\|x_s-x_k\|>0,\ \forall\ i\neq j,\ s\neq k.
 \end{displaymath}
Based on this result, for any $\pi(n)\in\Pi(n)$, the resulting EDM $D$ generated by the above $x_1,\dots, x_n$ is a feasible nonzero {solution} of $F(\pi(n))$, where the ordinal constraints in $\pi(n)$ actually hold with equality for such $D$. The proof is finished.
\end{proof}

\begin{theorem}\label{thm-2}(With Rank Constraint)
There exists at least one equivalent class of ordinal constraints $\mathcal O(n)$ such that for any $\pi(n)\in\mathcal O(n)$ {and any $r$}, $F(\pi(n),r)$ admits a nontrivial solution.
\end{theorem}

\begin{proof}

{
We can pick up $x_1,\dots, x_n\in\Re^r$ satisfying
\begin{displaymath}
	\|x_3-x_1\|\neq\|x_2-x_1\|.
\end{displaymath}
The resulting EDM $\overline D\in E(r)$ satisfies $\overline D_{12}\neq \overline D_{13}$. Now we rank the off-diagonal elements $\overline D_{ij}$ ($i<j$) in a nonincreasing way. Assume we obtain the following sequence
\begin{displaymath}
\overline D_{i_1j_1}\ge\overline D_{i_2j_2}\ge \dots\ge \overline D_{i_mj_m}.
\end{displaymath}
Then $\pi(n)=\{(i_1,j_1),\dots,(i_m,j_m)\}$ is a group of ordinal constraints that $\overline D$ satisfies.

With Lemma \ref{lem-1}, for any $\pi'(n)\sim \pi(n)$, $F(\pi'(n),r)$ admits a nontrivial feasible solution. The proof is finished. }
\end{proof}

\begin{theorem}\label{thm-3}
When $r\ge n-2$, $F(\pi(n),r)$ admits a nontrivial feasible solution for any $\pi(n)\in\Pi(n)$.
\end{theorem}

\begin{proof} See Appendix.
\end{proof}

{In fact, what we are more interested} in is the problem that $r< n-2$ (in especial $r\ll n$ such as $r=2,\ 3$). We would like to point out that for $r<n-2$, the result in Theorem \ref{thm-3} may fail. The counterexample is given as follows. When $n = 4,\ r = 1$, we found a group of ordinal constraints
	\[
	\pi(4)=\{(2,3),(1,2),(1,3), (1,4), (3,4), (2,4)\}
	\] such that $F(\pi(4), 1)$ admits only zero feasible solution.
	
	However, when $n=5$ and $r=2$, we can still construct {a nontrivial solution for some special cases of ordinal constraints}.
{\begin{theorem}\label{thm-4}	
	Given any $\pi(n)\in\Pi(n)$, assume that $\pi(n)$ take the form of (\ref{pi}). If either of the following condition holds, 
\bit
\item [(i)] $\{i_1,j_1\}\bigcap\{i_m,j_m\}=\emptyset$;
\item [(ii)] $\{i_1,j_1\}\bigcap\{i_m,j_m\}\neq\emptyset$ and $\{i_{m-1},j_{m-1}\}\bigcap\{i_m,j_m\}=\emptyset$;
\eit  
 $F(\pi(n),r)$  admits a nontrivial feasible solution.
\end{theorem}
\begin{proof}
(i) 
Without loss of generality, we assume $\{i_1,j_1\}=\{1,2\}$, $\{i_m,j_m\}=\{4,5\}$, that is, $D_{12}$ is  required to be the largest component in $D$ and $D_{45}$ the smallest. By setting point $4$ and point $5$ to coincide with each other,  we can find a nontrivial solution for any full ordinal constraints, as shown in the left part of Fig. \ref{fig:thm4}. Here  both  the triangle with vertices $1,3,4$ and triangle with vertices $2,3,4$ are regular triangles.  Such set of points leads to a nontrivial solution of $F(\pi(n),r)$ for any ordinal constraints $\pi(n)$ given by
  \[
  \{(1,2), (i_2,j_2), \dots, (i_{m-1}, j_{m-1}), (4,5))\}.
  \]
  (ii) Without loss of generality, we can assume that $\pi(n)$ take the form of 
  \be\label{eq-pi-n}
  \{(1,2), (i_2,j_2), \dots, (3, 4), (1,5))\}.
  \ee
 By setting points $1$ and $5$ to coincide with each other,  and $3$, $4$ to coincide with each other, the resulting points shown in the right part of  Fig. \ref{fig:thm4} leads to a nontrivial solution of $F(\pi(n),r)$ for any ordinal constraints $\pi(n)$ defined by (\ref{eq-pi-n}). The proof is finished. 
\end{proof}
}
\begin{figure}[htbp]
	\centering
	\includegraphics[height=2.3cm,width=4cm]{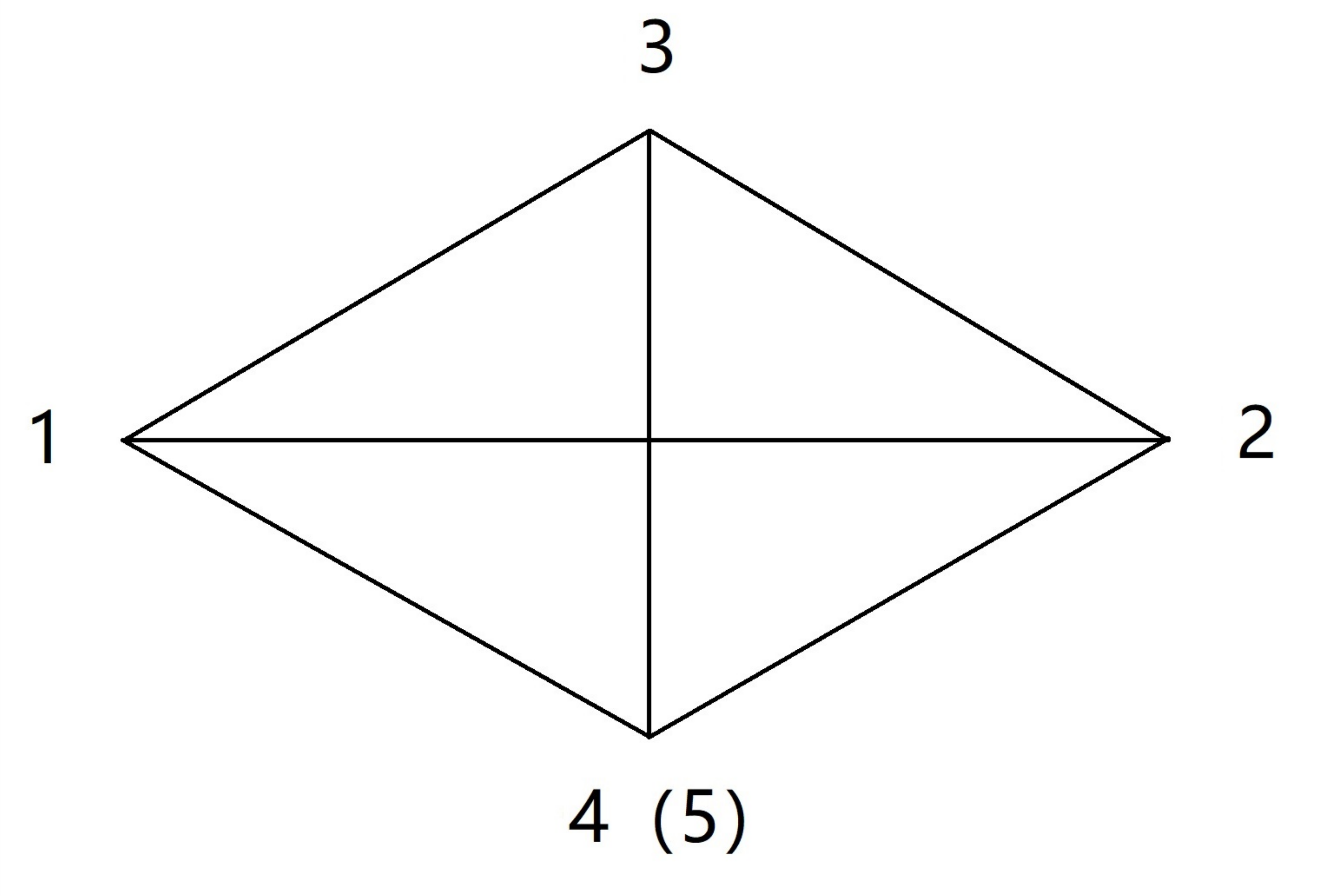}
	\includegraphics[height=3cm,width=3cm]{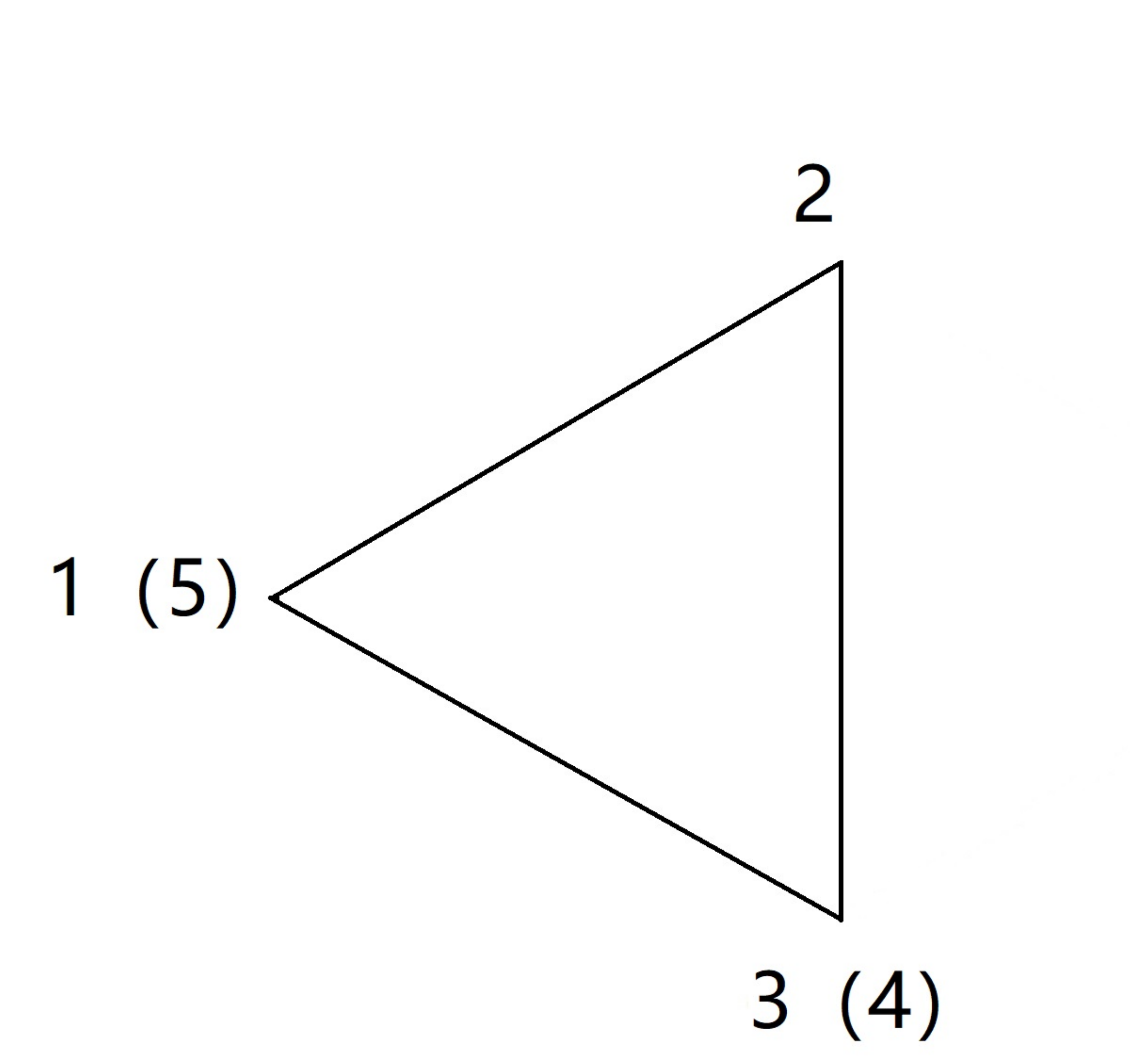}
	\caption{{Left: (i) in Theorem \ref{thm-4}; Right: (ii) in Theorem \ref{thm-4}. 
}}	
	\label{fig:thm4}
\end{figure}
{\begin{remark}
An open question is that for the case where  $\{i_1,j_1\}\bigcap\{i_m,j_m\}\neq\emptyset$ and $\{i_{m-1},j_{m-1}\}\bigcap\{i_m,j_m\}\neq\emptyset$, whether $F(\pi(n),r)$ still admits a nontrivial solution. 
\end{remark}
	
 
 On the other hand, Theorem \ref{thm-3} implies that $F(\pi(n),r)$ admits a nontrivial feasible solution for all $\pi(n)\in\Pi(n)$ when $r\ge n-2$, which inspires us to consider to divide $n$ points into subgroups. Let $\{N_k\}_{k\in I}$  be a partition of $\{1,\dots, n\}$, i.e., $N_i\cap N_j= \emptyset$ for any $i,j\in I$ and $\bigcup_{k\in I}\{N_k\}=\{1,\dots, n\}$. 

\begin{theorem}\label{thm-5}
	Given the embedding dimension $r$ and $n$ points with $r<n-2$. For $\tilde \pi(n)=\cup_{k\in I}\pi(N_k)$ with $|N_i|\le r+2$ and no overlaps between $N_k$, $F(\pi(N),r)$ admits a nontrivial feasible solution.
\end{theorem}
\begin{proof}
For any $\pi(N_k)$, due to $|N_k|\le r+2$, Theorem \ref{thm-3} implies that $F(\pi(N_k),r)$ admits a nontrivial feasible solution, $k\in I$. Since {$\{N_k\}_{k\in I}$ is a partition of $\{1,\ ,\dots,\ n\}$}, one can find points $x_1,\dots, x_n\in \Re^r$ such that the resulting nonzero EDM  satisfies the ordinal constraints in $F(\pi(N),r)$.  In other words, $F(\pi(N),r)$ admits a nontrivial feasible solution.
\end{proof}}

We end this part by the following {remark}.

{\bf Remark.}
{Theorem \ref{thm-3}, Theorem \ref{thm-4} and Theorem \ref{thm-5} partly {answer} question Q2.} As described above, for some $\pi(n)$ and some $r<n-2$, it is possible that EDMOC admits only zero solution. In the case of zero solution, it interprets the numerical observation 'crowding phenomenon', which means that all points collapse to one.

\section{A Majorized Penalty Approach}\label{sec-3}

In this part, we will discuss the majorized penalty approach for solving the EDMOC (\ref{prob-1}) with {squared weighted Frobenius norm}. That is,
\be\label{prob-2}
\begin{array}{ll}
\min_{D\in\mathcal S^n} & { \frac12\|W\circ(D-\Delta)\|^2}\\
\hbox{s.t.} & \diag(D)=0, \ -D\in \mathcal{K}^n_+(r),\\
&D_{i_1j_1} \ge D_{i_2j_2}\ge\dots\ge D_{i_mj_m},
\end{array}
\ee
where $\Delta\in\mathcal S^n$ is given, {and $W$} is the weight matrix with nonnegative elements. As we mentioned before, the challenges of solving {(\ref{prob-2})} lie in two aspects: (i) the nonconvex rank constraint and (ii) the potentially huge number of ordinal constraints. We will discuss the two issues in Section \ref{sec3-1} and Section \ref{sec3-2} {separately}. Details of the majrozation penalty approach {are} summarized in Section \ref{sec3-3}.

\subsection{Tackling Rank Constraint}\label{sec3-1}
To deal with the rank constraint, we make use of the majorized technique proposed in \cite{Zhou2017A,Qi2018}, which is detailed below.

Let $\Pi^B_{\mathcal{K}^n_+(r)}(D)$ denote the solution set of the following problem
\[
\min_{X\in \mathcal{K}^n_+(r)}\ \ {\frac12\|X-D\|^2}.
\]
Due to the nonconvexity of $\mathcal{K}_+^n(r)$, $\Pi_{\mathcal{K}^n_+(r)}^B(\cdot)$ may contain multiple solutions. Let {$\Pi_{\mathcal{K}_+^n(r)}(D)\in\Pi^B_{\mathcal{K}^n_+(r)}(D)$} be one of them.
It leads to the equivalent condition (\ref{g-rank}), 
by which problem (\ref{prob-2}) can be reformulated as the following problem
\be\label{prob-edm-6}
\begin{array}{ll}
\min_{D\in\mathcal S^n} & {\frac12\|W\circ(D-\Delta)\|^2}\\
\hbox{s.t.} & \diag(D)=0,\\
&D_{i_1j_1} \ge D_{i_2j_2}\ge\dots\ge D_{i_mj_m},\\
& g(D)=0.
\end{array}
\ee
The idea of majorized penalty approach is to penalize $g(D)$ into the objective function, and design a majorization approach to sequentially solve the penalty problem. This gives the majorized penalty approach.

As for the penalty problem, it takes the following form
\be\label{prob-edm-62}
\begin{array}{ll}
\min_{D\in\mathcal S^n} & {\frac12\|W\circ(D-\Delta)\|^2}+\rho g(D)\\
\hbox{s.t.} & \diag(D)=0,\\
&D_{i_1j_1} \ge D_{i_2j_2}\ge\dots\ge D_{i_mj_m},\\

\end{array}
\ee
where $\rho>0$ is the penalty parameter.

As in \cite{Zhou2017A,gao2009calibrating}, we design a majorization function of $g(D)$ in the following way. Recall that a majorization function of $g(D)$ at $D^k\in\mathcal S^n$, denoted as $g_m(D,D^k)$, has to satisfy following conditions
\be\label{major-f}
g_m(D^k,D^k) = g(D^k), \ \ g_m(D,D^k)\ge g(D), \ \forall \ D\in\mathcal S^n.
\ee
Based on the properties of $\Pi_{\mathcal{K}^n_+(r)}(\cdot)$ \cite{Zhou2017A}, there is
\[
g(D)=\frac12\|D\|^2-\frac12\|\Pi_{\mathcal{K}_+^n(r)}(-D)\|^2:=\frac12\|D\|^2-h(-D),
\]
and
\[
\Pi_{\mathcal{K}_+^n(r)}(D)\in\partial h(D),
\]
where $\partial h(D)$ is the set of subdifferentials of $h$ at $D$. Note that $h(D)$ is a convex function \cite{Mordukhovich2006Variational}, for any $V\in\partial h(D)$, there is
\be\label{sub-h}
h(\widehat D)-h(D)\ge\langle V,\widehat D-D\rangle, \ \forall \ \widehat D\in\mathcal S^n.
\ee With (\ref{sub-h}), we get the following majorization function of $g(D)$
\be\label{gm}
g_m(D,D^k) = \frac12\|D\|^2+\langle \Pi_{\mathcal{K}^n_+(r)}(-D^k),D-D^k \rangle.
\ee
It is easy to verify that $g_m(D,D^k)$ defined as in (\ref{gm}) satisfies properties of majorization function in (\ref{major-f}).
In other words, at each iteration $k$, we solve the following majorization subproblem
\be \label{EDM-Model-p2}
\begin{array}{ll}
	\min_{D\in\mathcal S^n} & \frac 12 \| W\circ(D-\Delta) \|^2 + \frac\rho2\|D\|^2 +\rho\langle\Pi_{\mathcal{K}_+^n(r)}(-D^k),D-D^k\rangle\\ [0.6ex]
	\mbox{s.t.} & \diag(D) = 0, \\
&D_{i_1j_1} \ge D_{i_2j_2}\ge\dots\ge D_{i_mj_m}.\\
\end{array}
\ee

After rearranging the terms in the objective function, we get the following subproblem{
\be \label{EDM-Model-p3}
\begin{array}{ll}
	\min_{D\in\mathcal S^n} & \frac {1}2 \|\widetilde{W}\circ( D - \widehat D^k) \|^2 \\
	\mbox{s.t.} & \diag(D) = 0, \\
&D_{i_1j_1} \ge D_{i_2j_2}\ge\dots\ge D_{i_mj_m},\\
\end{array}
\ee
where $\widetilde{W}_{ij}=(W_{ij}^2+\rho)^{\frac{1}{2}}$ and
$
\widehat D^k_{ij} = \frac{W_{ij}^2\Delta_{ij}-\rho(\Pi_{\mathcal{K}^n_+(r)}(-D^k))_{(i,j)}}{W_{ij}^2+\rho}.
$
}

We end this part by two remarks.

{{\bf Remark.}}
The remaining issue is how to calculate $\Pi_{\mathcal{K}_+^n(r)}(D)$. As shown in \cite[Eq.(22), Prop. 3.3]{Zhou2017A}, one particular $\Pi_{\mathcal{K}_+^n(r)}(D)$ can be computed through
\be\label{eq-projection}
\Pi_{\mathcal{K}_+^n(r)}(D) = PCA_r^+(JDJ)+(D-JDJ),
\ee
where
\be\label{part-eig}
PCA_r^+(A): = \sum_{i= 1}^r\max(0,\lambda_i)p_ip_i^T,
\ee
with the spectral decomposition of $A$ given by
\[
A = \lambda_1p_1p_1^T+\dots+\lambda_np_np_n^T, \
\]
$\lambda_1\ge\dots\ge\lambda_n$ are the eigenvalues of $A $ and $p_i$, $i = 1,\dots n$, are corresponding orthonormal eigenvectors.

{\bf Remark.}
The key point in the majorization function $g_m$ is the calculation of $\Pi_{\mathcal{K}^n_+(r)}(\cdot)$, where one can note from (\ref{part-eig}) that only the first $r$ leading eigenvalues are needed. It will significantly reduce the computational complexity when $n$ increases. That is the main difference {between our majorization function and the one} in \cite{Qi2014Computing,LiQi2017}, where the full spectral decomposition is used.

\subsection{Tackling Ordinal Constraints in Subproblems}\label{sec3-2}
To solve the subproblem {(\ref{EDM-Model-p3})}, note that the solution $D^{k+1}$ has the {zero diagonal elements}. The rest off-diagonal elements are given by solving the following subproblem

\be\label{D-sub-offd}
\begin{array}{ll}
	\min_{D_{ij}, i< j} & \sum_{i<j} \tilde{w}_{ij}(D_{ij}-\widehat D^k_{ij})^2\\
	\hbox{s.t.} &D_{i_1j_1} \ge D_{i_2j_2}\ge\dots\ge D_{i_mj_m}\ge0.\\
\end{array}
\ee
Here we {add} $D_{i_mj_m}\ge0$ to make the elements of $D$ nonnegative, which is also a necessary condition for EDM.

 Due to the symmetry of $D$, let \[
 x = (D_{i_1j_1},\dots, D_{i_mj_m})^T,\ y = (\widehat D^k_{i_1j_1},\dots, \widehat D^k_{i_mj_m})^T,
 \]
 we get the following subproblem
\be\label{isotonic}
\begin{array}{ll}
	\min_{x\in\Re^m}& \frac12\|\widetilde{H}(x-y)\|_2^2\\
	\hbox{s.t.}& x_1\ge x_2\ge \dots\ge x_m\ge0.
\end{array}
\ee
where
\[
\widetilde{H}:=\Diag(h)\ and\ h=(h_1,\dots, h_m)^T:=(\tilde w_{i_1j_1}, \dots, \tilde w_{i_mj_m})^T.
\]

This is the weighted isotonic regression problem which has been studied in \cite[P13]{barlow1972statistical}.

{To solve it, we first consider the special case with $\widetilde H = I$, which is the well-known isotonic regression}
 \be\label{isotonicR}
 \begin{array}{ll}
 	\min_{x\in\Re^m}& \frac12\|x-y\|_2^2\\
 	\hbox{s.t.}& x_1\ge x_2\ge \dots\ge x_m\ge0.
 \end{array}
 \ee
Problem (\ref{isotonicR})
can be solved by PAVA.
Here we modify the recent fast solver FastProxSL1 developed in \cite{Bogdan2015SLOPE} to solve (\ref{isotonicR}).	
FastProxSL1 \cite{Bogdan2015SLOPE} is used to solve the following problem
\be\label{sub2}
\begin{array}{ll}
	\min_{x\in\Re^m}& \frac12\|x-\hat y\|_2^2+\sum_{n=1}^n\lambda_ix_i\\
	\hbox{s.t.}& x_1\ge x_2\ge \dots\ge x_m\ge0,
\end{array}
\ee
where nonnegative and nonincreasing sequences $\hat y\in\Re^m and \ \lambda\in\Re^m $ are given. Problem (\ref{sub2}) can be reformulated as typical isotonic regression (\ref{isotonicR}) with $y=\hat y-\lambda$.	
 Consequently, we reach the following algorithm (denoted as \texttt{mFastProxSL1}) for solving {isotonic regression (\ref{isotonicR}).}

\begin{algorithm} \texttt{mFastProxSL1}
	\label{alg-isotonic}
		\bit
		\item [S0] Given $y\in\Re^m$. Let $x:=y$.
		\item [S1] While $x$ is not decreasing, do
		
		Identify strictly increasing subsequences, i.e. segments i:j such that\\
		$$x_i < x_{i+1} < \dots < x_j$$\\
		Replace the values of $x$ over such segments by its average value: for k $\in \{i,i+i,\dots,j\}$\\
		$$x_k \leftarrow \frac{1}{j-i+1}\sum_{i \le l \le j}x_{l}.$$
		\item [S2] Return $x=\max(x,0)\in\Re^m$.
		\eit
\end{algorithm}

As for the weighted case (\ref{isotonic}), 	we modify Algorithm \ref{alg-isotonic} by letting  \begin{equation*}
x_k\leftarrow\overline{x}=\frac{\sum_{i\le p\le j}h_p^2x_p}{\sum_{i\le p\le j}h_p^2}
\end{equation*}
in S1. {The resulting algorithm for (\ref{isotonic}) is} denoted as \texttt{w-mFastProxSL1}.

\subsection{Majorized Penalty Approach}\label{sec3-3}
Now we give the details of majorized penalty approach as shown in (\ref{MPA}). Similar as in \cite[Theorem 3.2]{Zhou2017A}, the majorized penalty approach enjoys the following convergence result.
	\begin{algorithm} Majorized Penalty Approach for {(\ref{prob-2})}
	\label{MPA}
	\bit
	\item [S0] Initialization. $\rho>0$, $\epsilon>0$. $D^0=0$, $k:=0$.
	\item [S1] Calculate $\Pi_{\mathcal{K}^n_+(r)}(-D^k)$. If $\|g(D^k)\|\le \epsilon$, stop. Otherwise, calculate $\widehat D^k$.
	\item [S2] Solve the subproblem {(\ref{isotonic})} via \texttt{mFastProxSL1} to get $D^{k+1}$.
	\item [S3] If the stopping criteria is satisfied, stop. Otherwise, increase $\rho$ and let $k:=k+1$, go to S1.
	\eit
\end{algorithm}
\begin{theorem}
Suppose $D^*$ is an optimal solution of (\ref{prob-2}). Let $D^*_\sigma$ be an optimal solution of the penalized problem (\ref{prob-edm-62}). Let $\epsilon>0$ be given. For any $\sigma\ge\sigma_{\epsilon}$, $D^*_{\sigma}$ must be $\epsilon$-optimal. That is,
\begin{displaymath}
	D^*_{\sigma}\in F(\pi(n),r), \ g(D^*_{\sigma})\le\epsilon \ \hbox{and } f(D^*_\sigma)\le f(D^*).
\end{displaymath}
\end{theorem}

{{\bf Remark.}}
Here we would like to highlight another advantage of the majorized {penalty} approach. As shown in Section \ref{sec2-3}, EDMOC (\ref{prob-2}) may {have} no feasible points for {some} $r$ and {some} ordinal constraints $\pi(n)$. In {that case}, solving the penalty problem (\ref{prob-edm-62}) seems to be a practical and good alternative. Our numerical test in Section \ref{sec5-2} will also verify this observation.

\section{Numerical Results}\label{sec-5}
In this section, we will conduct extensive numerical {tests} to demonstrate the efficiency of the proposed majorized penalty approach (denoted as \texttt{MPA}). We divide this section into four parts. In the first part, we discuss implementation issues of \texttt{MPA}. In the second part, we test the performance of \texttt{MPA}. In the third part, we compare \texttt{MPA} with some efficient solvers on SNL and MC. We also demonstrate numerical results for the weighted case in the last part.
\subsection{Implementations}\label{sec5-1}

The stopping criterion for \texttt{MPA} is the same as that in \cite{Zhou2017A}, that is,
\be\label{stop}
Fprog_k\le \epsilon_1, \ Kprog_k\le \epsilon_2,
\ee
where
\[
Fprog_k={\frac{f(D^{k-1})-f(D^k)}{\rho+f(D^{k-1})}, \ f(D) = \frac12\|D-\Delta\|^2},
\]
and
\[
Kprog_k = \frac{2g(D^k)}{\|JD^kJ\|^2}= 1-\frac{\sum_{i = 1}^r(\lambda_i^2-(\lambda_i-\max\{\lambda_i,0\})^2)}{\lambda_1^2+\dots+\lambda_n^2}
\]
We choose $\epsilon_1=\epsilon_2=10^{-3}$.
Other parameters are set as default. For solving subproblems by \texttt{mFastProxSL1} and \texttt{w-mFastProxSL1}, we modify the FastProxSL1.c\footnote{The code can be downloaded from {http://www-stat.stanford.edu/candes/SortedSL1}} file into the isotonic regression solver and the weighted isotonic regression solver, then use the mex file in Matlab.

After running \texttt{MPA}, we {adopt} cMDS to get the embedding points. These points will be {transformed} through Procrustes process to get the estimated points. Then we apply refinement step \cite{Fang2013Using} {to get the final estimation points and calculate RMSD and rRMSD to measure the error of nonrefined points and refined points separately.} The whole process is summarized as follows.

	\bit\item [S1]	Run \texttt{MPA} to get $D$.	
			\item [S2] Apply cMDS to get $x_1, \dots, x_n\in\Re^r$.
			\item [S3] {Apply Procrustes process to $x_1,\dots, x_n$ to get estimation points $\bar x_1, \dots, \bar x_n$. Calculate RMSD by
\[
			RMSD = \sqrt{\frac1n\sum_{i = 1}^n\|x_i^*-\bar x_i\|^2},
			\]
where $x^*_i, \dots, x^*_n$ are the true positions.}
			\item [S4] Apply Refinement Step to get final refined points $\hat x_1,\dots, \hat x_n$. Calculate rRMSD as above with $\bar x_i$ replaced by $\hat x_i$, $i = 1,\dots, n$.

\eit

All the tests are conducted by using Matlab R2016b on a computer with Intel (R) Core (TM) i5-6300HQ CPU @ 2.30GHz 2.30GHz, RAM 4GB.

\subsection{Performance Test}\label{sec5-2}

In this part, we test the performance of our algorithm. The test problem is generated in the following way. {A set of} $n$ points $x_1,\dots, x_n\in\Re^s$ are randomly generated to build an EDM $D$. {That is, $D_{ij}=\|x_i-x_j\|^2$, $i,j = 1,\dots, n$.} 
Here, we use $s$ to denote the dimension of points that generate $D$. We choose $s=11$ {and set $W_{ij}=1$ for all $i$ and $j$}. {The ordinal constraints are generated by $D$ in the following way.
\begin{algorithm} \texttt{Generating Ordinal Constraints}
	\label{alg-order}
\bit
\item [(a)] Input an EDM $D$. 
\item [(b)] Rank the elements $D_{ij}$, $i<j$, in a nonincreasing order as
\be
D_{j_1l_1}\ge D_{j_2l_2}\ge\dots\ge D_{j_ml_m}.
\ee
\item [(c)] Output ordinal constraints
\[
P_O=\{D\in\mathcal S^n \ | \ D_{j_1l_1}\ge D_{j_2l_2}\ge\dots\ge D_{j_ml_m}\}.\]
\eit
\end{algorithm}}
We set different prescribed embedding dimension $r$ in our test. The following information {is} reported in Table \ref{tab:Change_dim}: 
the size of $D$ $n$, the size of subproblem $m = \frac{n(n-1)}{2}$; the prescribed embedding dimension $r$; the cputime $t$ (in second, {including the refinement step}), the cputime for solving subproblem $t_{sub}$, cputime for partial spectral decomposition $t_{eig}$ used in (\ref{part-eig}); the number of iterations $Iter$, as well as RMSD, rRMSD, $Kprog_k$, $Fprog_{k}$, which have already been defined.

\begin{table}[tbhp]
	{\footnotesize
		\caption{Results for different embedding dimension $r$.} \label{tab:Change_dim}
		\begin{center}
			\resizebox{\textwidth}{50mm}{
				\begin{tabular}{|c|c|c|c|c|c|c|c|c|c|c|c|c|c|} \hline
					n & m & r & t(s) & $t_{sub}$(s) & $t_{eig}$(s) & RMSD & rRMSD & $Kprog_k$ & $Fprog_k$ & Iter \\ \hline	
					500 & 124750 &  2 & 3.03 & 1.01 & 0.87 & 50.02 & 4.57 & 9.22e-1 & 2.49e-3 & 76 \\\hline
					500 & 124750 &  3 & 3.79 & 1.35 & 1.12 & 40.27 & 3.40 & 9.01e-1 & 2.20e-3 & 93 \\\hline
					500 & 124750 &  4 & 4.62 & 1.73 & 1.30 & 41.02 & 3.14 & 8.96e-1 & 2.17e-3 & 137 \\\hline
					500 & 124750 &  5 & 5.24 & 2.01 & 1.50 & 42.19 & 3.14 & 8.31e-1 & 2.10e-3 & 154 \\\hline
					500 & 124750 &  6 & 5.89 & 2.32 & 1.38 & 42.51 & 3.72 & 7.32e-1 & 2.23e-3 & 181 \\\hline
					500 & 124750 &  7 & 6.84 & 2.43 & 1.68 & 43.75 & 3.21 & 6.23e-1 & 2.13e-3 & 201 \\\hline
					500 & 124750 &  8 & 7.98 & 3.21 & 1.85 & 43.80 & 3.22 & 5.17e-1 & 2.39e-3 & 228 \\\hline
					500 & 124750 &  9 & 8.38 & 3.14 & 2.07 & 41.21 & 3.82 & 3.18e-1 & 2.11e-4 & 258 \\\hline
					500 & 124750 & 10 & 12.07 & 4.90 & 2.98 & 44.48 & 4.12 & 1.87e-1 & 1.87e-5 & 347 \\\hline
					1000 & 499500 &  2 & 11.45 & 4.10 & 3.60 & 50.78 & 3.21 & 9.92e-1 & 1.32e-3 & 66 \\\hline
					1000 & 499500 &  3 & 13.22 & 5.16 & 4.16 & 46.35 & 2.32 & 9.76e-1 & 1.32e-3 & 73 \\\hline
					1000 & 499500 &  4 & 15.24 & 5.86 & 4.66 & 45.62 & 2.09 & 9.75e-1 & 1.53e-3 & 86 \\\hline
					1000 & 499500 &  5 & 16.94 & 6.86 & 4.94 & 44.98 & 2.01 & 9.55e-1 & 1.48e-3 & 95 \\\hline
					1000 & 499500 &  6 & 20.54 & 8.29 & 6.03 & 44.97 & 2.06 & 9.24e-1 & 1.41e-3 & 115 \\\hline
					1000 & 499500 &  7 & 23.59 & 9.91 & 6.32 & 45.20 & 2.15 & 8.65e-1 & 1.21e-3 & 142 \\\hline
					1000 & 499500 &  8 & 29.87 & 12.67 & 7.73 & 45.41 & 2.25 & 7.53e-1 & 8.24e-4 & 168 \\\hline
					1000 & 499500 &  9 & 47.11 & 18.59 & 12.53 & 45.64 & 2.34 & 5.64e-1 & 2.90e-4 & 238 \\\hline
					1000 & 499500 & 10 & 56.10 & 23.85 & 13.36 & 45.81 & 2.42 & 3.20e-1 & 3.26e-5 & 347 \\\hline
					2000 & 1999000 &  2 & 54.84 & 20.00 & 19.09 & 50.03 & 4.32 & 9.96e-1 & 7.98e-4 & 47 \\\hline
					2000 & 1999000 &  3 & 63.05 & 23.11 & 23.57 & 46.61 & 2.96 & 9.93e-1 & 7.97e-4 & 77 \\\hline
					2000 & 1999000 &  4 & 72.34 & 25.90 & 28.12 & 45.56 & 2.35 & 9.88e-1 & 8.09e-4 & 78 \\\hline
					2000 & 1999000 &  5 & 78.14 & 29.31 & 28.21 & 45.29 & 2.06 & 9.80e-1 & 8.11e-4 & 90 \\\hline
					2000 & 1999000 &  6 & 90.74 & 35.46 & 30.45 & 45.33 & 1.92 & 9.65e-1 & 8.01e-4 & 112 \\\hline
					2000 & 1999000 &  7 & 103.75 & 41.57 & 33.59 & 48.63 & 1.92 & 9.37e-1 & 7.18e-4 & 132 \\\hline
					2000 & 1999000 &  8 & 129.68 & 53.52 & 38.83 & 91.85 & 1.88 & 8.84e-1 & 5.54e-4 & 172 \\\hline
					2000 & 1999000 &  9 & 168.21 & 71.08 & 47.57 & 46.98 & 1.86 & 7.73e-1 & 2.46e-4 & 234 \\\hline
					2000 & 1999000 & 10 & 234.03 & 108.50 & 60.20 & 46.15 & 1.81 & 5.58e-1 & 3.43e-5 & 332 \\\hline
				\end{tabular}}
			\end{center}
		}
	\end{table}

It can be seen that as $r$ increases from $2$ to $10$, it takes more iterations, leading to more cputime. The cputime is mainly spent on subproblems and partial spectral decomposition. For each test, the partial spectral decomposition takes a bit less cputime than solving the subproblem, both of which increases slowly as $n$ grows. This verifies our claim that the partial {spectral} decomposition in calculating $\Pi_{\mathcal{K}^n_+(r)}(\cdot)$ has lower computational complexity than the full spectral decomposition. From $Kprog_k$ and $Fprog_k$, it can be observed that the stopping criteria is hardly satisfied. {This can  be partly explained by the feasibility of {EDMOC} (\ref{prob-2}). In other words, for  $r<n-2$, it is possible that  EDMOC may not have a nontrivial feasible solution. From the numerical point of view, it means that for a test problem with $r<n-2$, if the ordinal constraints are added randomly, it may be difficult for the algorithm to find a nontrivial solution, let alone to find a nontrivial feasible solution satisfying the stopping criteria.} 

\subsection{Applications}\label{sec5-3}
\

{\bf{Sensor Network Localization.}} One typical application of {EDMOC} is the sensor network localization problem, where the positions of some points are known (referred {to as} anchors), and the rest are unknown (referred {to as} sensors). We test Square Network which is widely tested \cite{Biswas2006Semidefinite}.  In our following test, we only consider the situation without anchors, {that is, $m=0$}.
$\Delta$ is generated in the same way as \cite{Zhou2017A,Bai2015Tackling}. Specifically,
the generation of the n sensors $(x_1,\dots,x_n)$ follows the uniform distribution over
the square region $[-0.5,0.5]\times [-0.5,0.5]$. {The element $\delta_{ij}$ in $\Delta$ is given by}
\be\label{eq-delta}\delta_{ij} := \|x_i-x_j\|\times |1+\epsilon_{ij}\times nf|,\ \ \forall (i, j) \in \mathscr{N}_x; \ \delta_{ij}=0, \hbox{ otherwise}\ee and
\be\mathscr{N}_x:=\{(i,j)\ |\ \|x_i-x_j\|\le R,\ i>j>m\},\ee
where $R$ is known as the radio range, $\epsilon_{ij}$ are independent standard normal random variables and $nf$ is the noise factor (see \cite{Zhou2017A}). 
This type of perturbation in $\delta_{ij}$ is known to be multiplicative and follows the unit-ball rule in defining $\mathscr{N}_x$. {Denote $$\Delta_{rate}:=\frac{the\ number\ of\ nonzero\ entries\ in\ \Delta}{the\ number\ of\ all\ entries\ in\ \Delta}$$ as the measure of {density}.}
{The ordinal constraints are added in the following way. First, we calculated the EDM $D$ based on $x_1, \dots, x_n$. Then we run Alg. \ref{alg-order} to get the ordinal constraints. By doing this, we can  guarantee the test problem to have a nontrivial feasible solution $D$.} 

{We select the well-known \texttt{SMACOF} \cite{Cox,Borg2010Modern,deLeeuw2009}, \texttt{SQREDM} \cite{Zhou2017A} and the Inexact Smoothing Newton Method (\texttt{ISNM}) \cite{LiQi2017}   for comparison due to their high-quality code and availability. \texttt{SMACOF} is a traditional method in dealing with MDS and NMDS and has a high reputation in experimental sciences. We use the enhanced implementation of \texttt{SMACOF} \cite{SMACOF}\footnote{The code can be downloaded from {http://tosca.cs.technion.ac.il}}. \texttt{ISNM} was proposed to solve convex EDM problems with only  ordinal constraints.  The latest \texttt{SQREDM} has superior performance than other methods in terms of both the speed and the accuracy as shown in \cite{Zhou2017A}.   The parameters in the four methods are set as follows.  For {each method}, 
	the weights are chosen as $W_{ij}=1$ if $\delta_{ij}>0$, otherwise, $W_{ij}=0$. In \texttt{SMACOF}, we set $rtol = 10^{-2}$, $iter = 10^{3}$ and its initial point is given by cMDS on $\Delta$. In \texttt{MPA} and \texttt{SQREDM}, we use the same stopping criteria as in (\ref{stop}) with $\epsilon_1=\epsilon_2=10^{-3}$ {and the minimum iterations 10}. Since \texttt{ISNM} solves a system of smoothing equations sequentially, so it has { the} different stopping criteria. We set comparable stopping criteria in the order of $10^{-3}$ {and set maximum iterations 100} to make the comparison reasonable. The ordinal constraints are generated by Alg. \ref{alg-order}. Other  parameters in \texttt{SQREDM}, \texttt{ISNM} and \texttt{SMACOF} are set as default.}

{To visualize data, we test {SNL example with $s=2$ and the embedding dimension $r=2$.  Recall that $nf$ defined as in (\ref{eq-delta}) is chosen as $nf=0.1$, corresponding to $10\%$ noise level.}  $R$ is chosen as $1.4$, $1$ and $0.2$. {We test the case of no anchors. For stability, we run each test 10 times and report the average results.} 
	
 For $R=1.4$, the results are shown in Table \ref{tab:SNL_R1.4}. 
\texttt{MPA}, \texttt{SMACOF} and \texttt{SQREDM} are much faster than \texttt{ISNM} (denoted as A1, A2, A3, A4 respectively), {whereas} \texttt{MPA} and \texttt{SMACOF} is slightly faster than \texttt{SQREDM}.
 For RMSD and rRMSD, \texttt{MPA} performs slightly better than \texttt{SMACOF}, \texttt{ISNM} and \texttt{SQREDM}. This is reasonable since \texttt{MPA} solves the model (\ref{prob-2}) whereas \texttt{ISNM} solves the convex relaxation problem (\ref{prob-edm-3}). Compared with \texttt{SQREDM}, \texttt{MPA} solves a different model, with ordinal constraints rather than bound constraints, which provide more information of the embedding points. In terms of rRMDS, it seems that the refinement step does not help for \texttt{MPA}, but indeed improves the performance of \texttt{SMACOF}, \texttt{ISNM} and \texttt{SQREDM}. {We can conclude that when the density $\Delta_{rate}$ is high, \texttt{MPA}, \texttt{SMACOF} and \texttt{SQREDM} can provide high quality solution in short time. }
Typical embedding results are demonstrated in Fig. \ref{fig:SNL_R1.4} with $R=1.4$ and $n=200$, where sensors $\{x_i\}$ in pink points are jointed to their corresponding true locations (blue circles).

\begin{table}[tbhp]
	{\footnotesize
		\caption{Results on SNL by four methods with $R=1.4$
} \label{tab:SNL_R1.4}
		\begin{center}
			\resizebox{\textwidth}{12mm}{
				\begin{tabular}{|c|c|c|c|c|c|c|} \hline
					\multirow{2}{*}{n}&\multirow{2}{*}{m}&\multirow{2}{*}{$\Delta_{rate}$} & t(s) & RMSD & rRMSD & Iter\\ \cline{4-7}
					~ &~&~ & A1$|$A2$|$A3$|$A4 & A1$|$A2$|$A3$|$A4 & A1$|$A2$|$A3$|$A4 & A1$|$A2$|$A3$|$A4 \\ \hline
					200 & 19900 &99.5\% & 0.15$|$0.12$|$0.29$|$40.82 & 3.7e-4$|$5.0e-2$|$1.7e-2$|$4.6e-2 & 1.5e-2$|$1.5e-2$|$1.5e-2$|$1.5e-2 &  10$|$  3$|$ 10$|$63.1 \\
					400 & 79800 &99.8\% & 0.76$|$0.75$|$1.15$|$434.07 & 1.3e-4$|$4.8e-2$|$1.2e-2$|$4.6e-2 & 1.1e-2$|$1.1e-2$|$1.1e-2$|$1.1e-2 &  10$|$  3$|$ 10$|$100 \\
					600 & 179700 &99.8\% & 3.08$|$4.63$|$3.87$|$1243.19 & 7.2e-5$|$4.8e-2$|$1.0e-2$|$4.6e-2 & 8.7e-3$|$8.7e-3$|$8.7e-3$|$8.7e-3 &  10$|$  3$|$ 10$|$100 \\
					800 & 319600 &99.9\%& 3.37$|$3.27$|$4.84$|$2783.22 & 4.5e-5$|$4.7e-2$|$8.9e-3$|$4.6e-2 & 7.4e-3$|$7.4e-3$|$7.4e-3$|$7.5e-3 &  10$|$  3$|$ 10$|$100 \\
					1000 & 499500 &99.9\% & 4.27$|$5.66$|$7.72$|$4748.70 & 3.4e-5$|$4.7e-2$|$8.1e-3$|$4.5e-2 & 6.8e-3$|$6.8e-3$|$6.8e-3$|$6.9e-3 &  10$|$  3$|$ 10$|$100 \\
					1500 & 1124250 &99.9\% & 9.56$|$8.66$|$14.43$|$- & 1.9e-5$|$4.7e-2$|$6.7e-3$|$- & 5.4e-3$|$5.4e-3$|$5.4e-3$|$- &  10$|$  3$|$ 10$|$- \\
					2000 & 1999000 &99.9\% & 12.87$|$17.23$|$24.31$|$- & 1.2e-5$|$4.7e-2$|$6.0e-3$|$- & 4.8e-3$|$4.8e-3$|$4.8e-3$|$- &  10$|$  3$|$ 10$|$-  \\ \hline
				\end{tabular}}
			\end{center}}
		\end{table}
\begin{table}[tbhp]
	{\footnotesize
		\caption{Results on SNL by four methods with $R=1.0$
		} \label{tab:SNL_R1}
		\begin{center}
			\resizebox{\textwidth}{12mm}{
				\begin{tabular}{|c|c|c|c|c|c|c|} \hline
					\multirow{2}{*}{n}&\multirow{2}{*}{m}&\multirow{2}{*}{$\Delta_{rate}$} & t(s) & RMSD & rRMSD & Iter\\ \cline{4-7}
					~ &~&~ & A1$|$A2$|$A3$|$A4 & A1$|$A2$|$A3$|$A4 & A1$|$A2$|$A3$|$A4 & A1$|$A2$|$A3$|$A4 \\ \hline
					200 & 19900 &97.1\% & 0.20$|$0.22$|$0.33$|$88.05 & 3.0e-3$|$1.8e-1$|$1.8e-2$|$1.1e-1 & 1.5e-2$|$8.5e-2$|$1.5e-2$|$1.6e-2 &  10$|$4.1$|$ 10$|$100 \\
					400 & 79800 &97.2\% & 0.86$|$1.29$|$1.19$|$557.94 & 2.4e-3$|$1.4e-1$|$1.5e-2$|$1.4e-1 & 1.0e-2$|$5.1e-2$|$1.0e-2$|$1.1e-2 &  10$|$4.5$|$ 10$|$100 \\
					600 & 179700 &97.4\% & 2.75$|$4.71$|$3.74$|$1214.40 & 2.0e-3$|$1.2e-1$|$1.3e-2$|$1.2e-1 & 8.7e-3$|$5.4e-2$|$8.7e-3$|$8.8e-3 &  10$|$4.5$|$ 10$|$100 \\
					800 & 319600 &97.4\% & 3.82$|$6.70$|$4.87$|$2692.03 & 1.9e-3$|$1.2e-1$|$1.3e-2$|$1.4e-1 & 7.5e-3$|$4.1e-2$|$7.5e-3$|$7.8e-3 &  10$|$4.8$|$ 10$|$100 \\
					1000 & 499500 &97.4\% & 6.14$|$14.24$|$7.98$|$5152.36 & 1.8e-3$|$1.1e-1$|$1.4e-2$|$1.2e-1 & 6.8e-3$|$5.3e-2$|$6.8e-3$|$6.6e-3 &  10$|$  5$|$ 10$|$100 \\
					1500 & 1124250 &97.5\% & 10.34$|$16.49$|$14.27$|$- & 1.5e-3$|$1.0e-1$|$1.4e-2$|$- & 5.5e-3$|$4.5e-2$|$5.5e-3$|$- &  10$|$  5$|$ 10$|$ - \\
					2000 & 1999000 &97.5\% & 18.69$|$26.88$|$26.11$|$- & 1.4e-3$|$1.0e-1$|$1.4e-2$|$- & 4.8e-3$|$4.2e-2$|$4.8e-3$|$- &  10$|$  5$|$ 10$|$ - \\ \hline
				\end{tabular}}
			\end{center}}
		\end{table}
		\begin{table}[tbhp]
			{\footnotesize
				\caption{Results on SNL by four methods with $R=0.2$} \label{tab:SNL_R0.2}
				\begin{center}
					\resizebox{\textwidth}{9mm}{
						\begin{tabular}{|c|c|c|c|c|c|c|} \hline
							\multirow{2}{*}{n}&\multirow{2}{*}{m}&\multirow{2}{*}{$\Delta_{rate}$}& t(s) & RMSD & rRMSD & Iter\\ \cline{4-7}
							~ &~&~ & A1$|$A2$|$A3$|$A4 & A1$|$A2$|$A3$|$A4 & A1$|$A2$|$A3$|$A4 & A1$|$A2$|$A3$|$A4 \\ \hline
							100 & 4950 &10.3\% & 12.79$|$0.17$|$7.72$|$158.65 & 2.7e-3$|$4.1e-1$|$1.7e-1$|$4.1e-1 & 7.2e-2$|$4.2e-1$|$1.6e-1$|$3.9e-1 & 1259.8$|$  2$|$796.8$|$100  \\
							200 & 19900 &10.6\%  & 32.03$|$0.47$|$21.84$|$359.57 & 3.2e-3$|$4.1e-1$|$2.4e-1$|$4.0e-1 & 5.3e-2$|$4.2e-1$|$2.4e-1$|$4.2e-1 & 2000$|$  2$|$1468.7$|$100  \\
							300 & 44850 &10.4\%  & 55.58$|$0.74$|$37.61$|$600.05 & 3.3e-3$|$4.1e-1$|$2.3e-1$|$4.2e-1 & 4.1e-2$|$4.2e-1$|$2.1e-1$|$4.4e-1 & 2000$|$  2$|$1407.5$|$100  \\
							400 & 79800 &10.3\%  & 75.10$|$1.16$|$58.39$|$1019.48 & 2.3e-1$|$4.1e-1$|$2.1e-1$|$4.0e-1 & 2.4e-1$|$4.1e-1$|$2.0e-1$|$4.5e-1 & 1529.3$|$  2$|$1374.3$|$100  \\
							500 & 124750 &10.5\%  & 91.15$|$1.80$|$71.94$|$1565.95 & 3.1e-1$|$4.1e-1$|$1.9e-1$|$4.1e-1 & 3.2e-1$|$4.2e-1$|$1.4e-1$|$4.0e-1 & 1210.8$|$  2$|$1199.6$|$100 \\ \hline
						\end{tabular}}
					\end{center}}
				\end{table}
{For $R=1.0$, as shown in Table \ref{tab:SNL_R1}, \texttt{MPA} and \texttt{SQREDM} can provide reasonably good embedding results.} After the refinement step, \texttt{SMACOF}'s and \texttt{ISNM}'s rRMSD become acceptable. {For $R=0.2$, most of the dissimilarity information is missing.}  Table \ref{tab:SNL_R0.2} {demonstrates that only \texttt{MPA}} can provide  good embedding result. It is easy to understand that smaller $R$ leads to more {computational time  and larger number of} iterations. 

To see the effect of $R$ in four methods, we increase $R$ from $0.2$ to $1.4$ by fixing $n=200$. The results in Fig.\ref{fig:SNL_R} demonstrate the trends of RMSD, rRMSD and Time. Both \texttt{MPA} and \texttt{SQREDM} are winners in terms of Time, RMSD and rRMSD. When $R$ is  small ($R<0.6$), only \texttt{MPA} can perform well. 

 Note that both \texttt{SQREDM} and \texttt{MPA} use the majorization technique and singular value decomposition. To give a further comparison, we report more details about  computational time of the two methods in Table \ref{tab:SNL_time} and Table \ref{tab:SNL_time_R50} with different $R$ (namely $R=140$ and $R=50$ respectively) and $nf=0.1$ in bigger square region $[-50,50]\times [-50,50]$ which allows us to test for larger number of points $n$. $Ave_{eig}$ is the average time per iteration for partial singular value decomposition and $Ave_{sub}$ is the average time per iteration for solving subproblem. {One may notice that in Table \ref{tab:SNL_time_R50} when $n$ is large, the cuptime for solving subproblems and partial singular decomposition only takes about $10\%$ of the total cuptime for both \texttt{MPA} and \texttt{SQREDM}. The reason is that when $n$ is large and density is medium, even getting a good starting point $D^0$ with all elements available spends large amount of time. Both ordering elements of $D^k$ and reordering back take time as well. Combining with Fig. \ref{fig:histogram}, as we can see, due to $\Delta_{rate}<90\%$, both two methods should call \texttt{graphallshortestpaths($\cdot$)} to determine whether the neighborhood graph of $\Delta$ is connected which dominates most of time.} \texttt{MPA} is faster than \texttt{SQREDM} in terms of total cputime. The two methods take comparable time for partial singular value decomposition, as demonstrated by $Ave_{eig}$ and $t_{eig}$. However, \texttt{MPA} takes less time in solving subproblem. {This can be explained by the fact that the computational complexity for solving subproblem of \texttt{MPA} is $O(\frac{n(n-1)}{2})$ , whereas that for the subproblem of \texttt{SQREDM} is $O(n^2s)$, where $s$ is the componentwise time complexity including basic operations and calling \texttt{cos($\cdot$)}, \texttt{arccos($\cdot$)}.}

\begin{figure}[htbp]
	\centering
	\subfigure[MPA]{
		\label{fig:SNL_EDMEDMOC}
		\includegraphics[height=2.5cm,width=5.7cm]{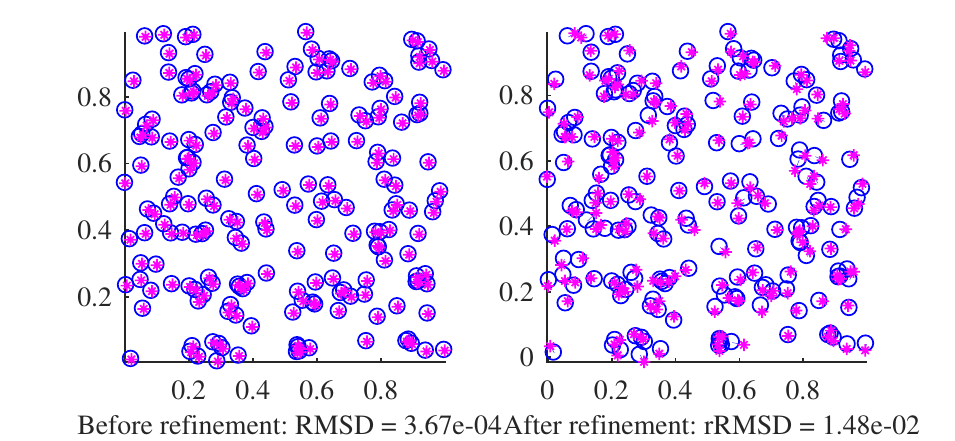}}
	\subfigure[SMACOF]{
		\label{fig:SNL_SMACOF}
		\includegraphics[height=2.5cm,width=5.7cm]{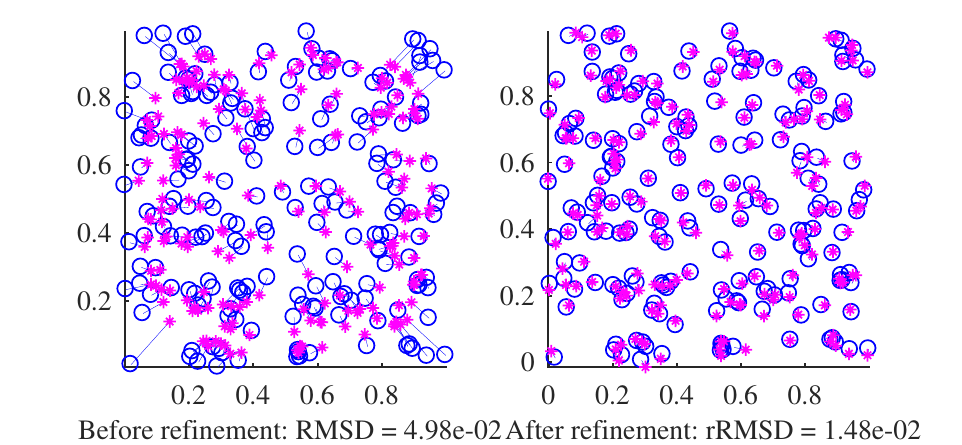}}
	\subfigure[SQREDM]{
		\label{fig:SNL_SQREDM}
		\includegraphics[height=2.5cm,width=5.7cm]{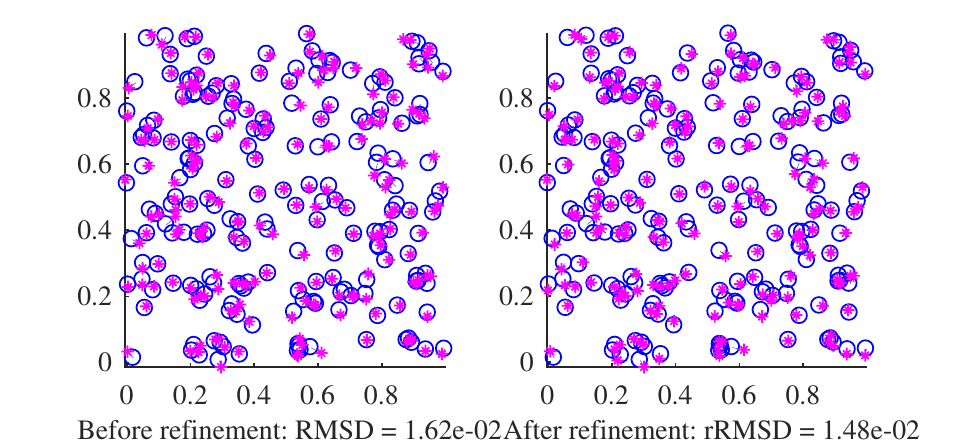}}
	\subfigure[ISNM]{
		\label{fig:SNL_ISNM}
		\includegraphics[height=2.5cm,width=5.7cm]{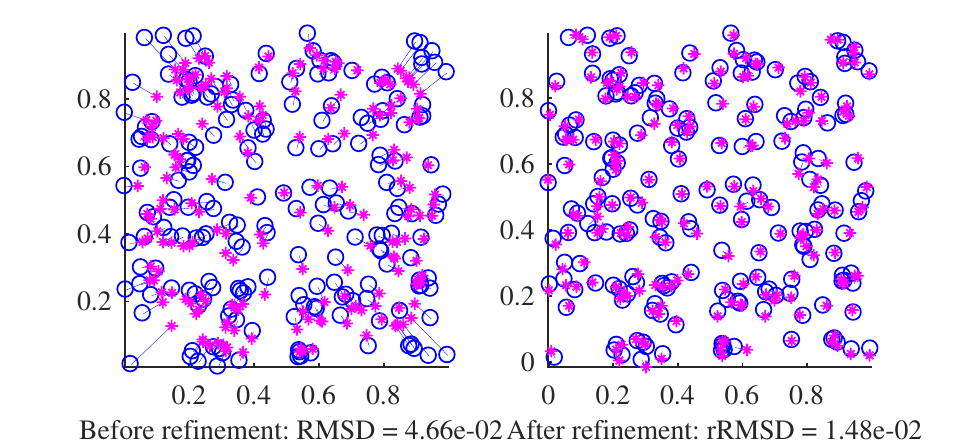}}
	\caption{Localization by four methods with n=200}	
	\label{fig:SNL_R1.4}
\end{figure}

\begin{figure}[htbp]
	\centering
	\includegraphics[height=3.5cm,width=3.5cm]{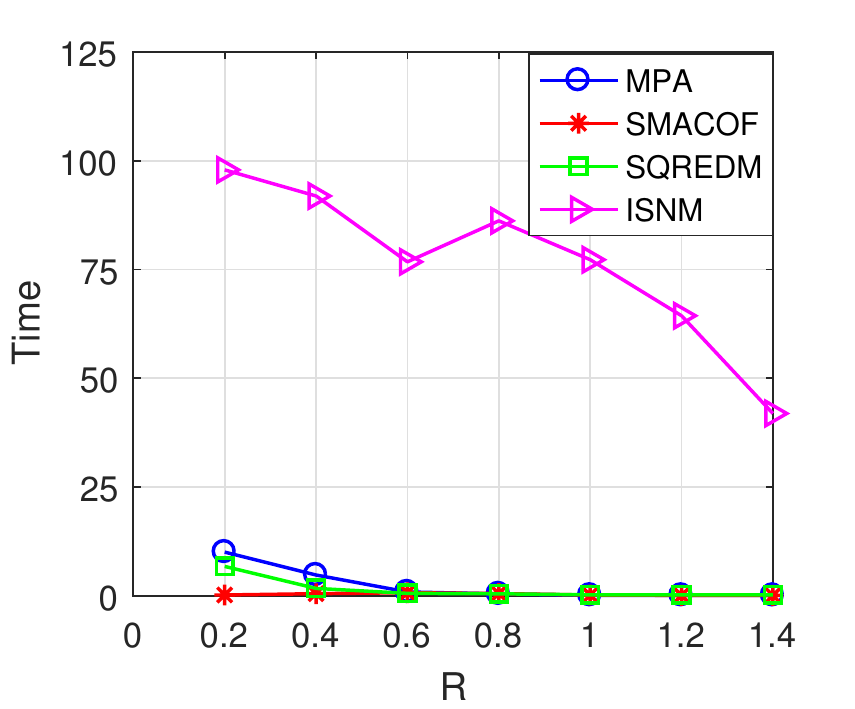}
	\includegraphics[height=3.5cm,width=3.5cm]{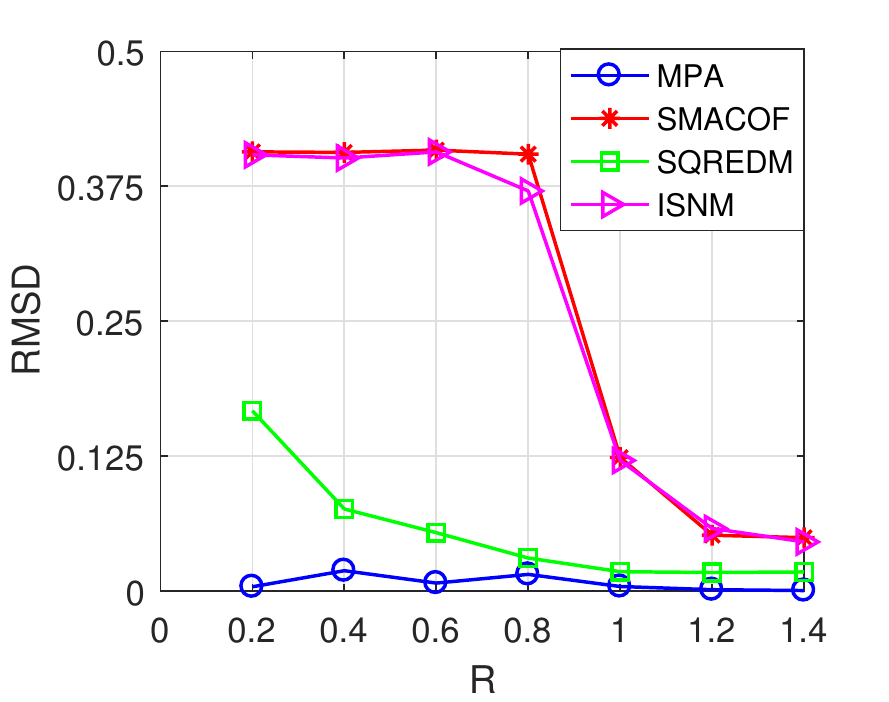}
	\includegraphics[height=3.5cm,width=3.5cm]{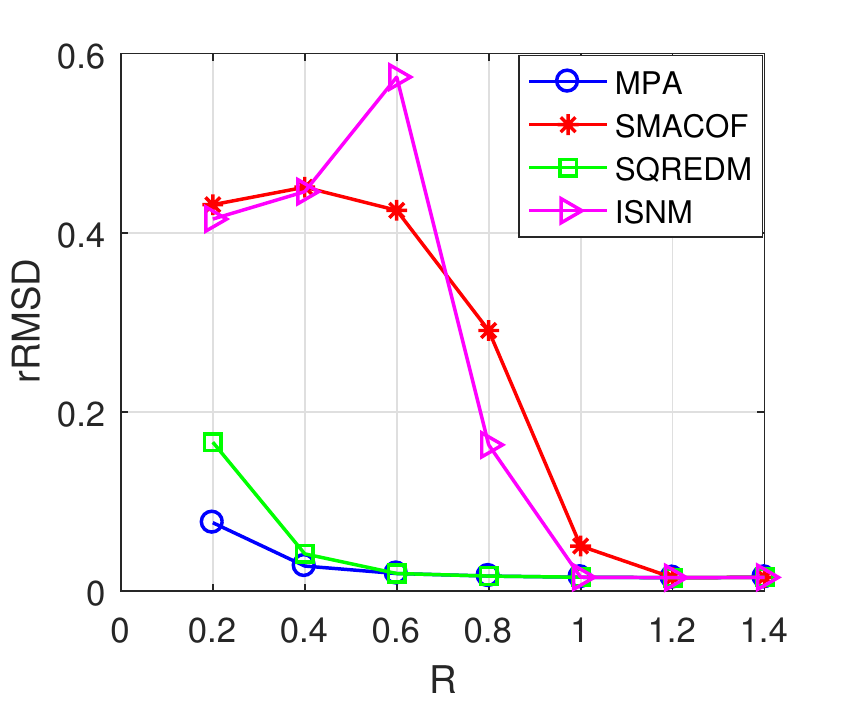}
	\caption{Comparisons of four methods with different $R$}	
	\label{fig:SNL_R}
\end{figure}
\begin{table}[tbhp]
	{\footnotesize
		\caption{The main time cost on SNL by \texttt{MPA} and \texttt{SQREDM} with $R=140$} \label{tab:SNL_time}
		\begin{center}
			\resizebox{\textwidth}{11mm}{
				\begin{tabular}{|c|c|c|c|c|c|c|c|c|c|} \hline
					\multirow{2}{*}{n}&\multirow{2}{*}{$\Delta_{rate}$}& t(s) &$t_{sub}(s)$&$t_{eig}(s)$&$Ave_{sub}$&$Ave_{eig}$& RMSD & rRMSD & Iter\\ \cline{3-10}
					~  &~  & A1$|$A3 & A1$|$A3 & A1$|$A3 & A1$|$A3& A1$|$A3& A1$|$A3& A1$|$A3& A1$|$A3 \\ \hline
		1000 &99.9\%& 7.99$|$11.63 & 0.95$|$4.30 & 5.33$|$5.31 & 0.09$|$0.43 & 0.53$|$0.53 & 1.7e-4$|$7.8e-3 & 6.7e-3$|$6.7e-3 & 10$|$10 \\
		1500 &99.9\%& 13.69$|$20.14 & 2.39$|$8.54 & 7.68$|$7.46 & 0.24$|$0.85 & 0.77$|$0.75 & 8.7e-5$|$6.4e-3 & 5.5e-3$|$5.5e-3 & 10$|$10 \\
		2000 &99.9\%& 19.91$|$32.27 & 4.96$|$16.10 & 9.04$|$8.87 & 0.50$|$1.61 & 0.90$|$0.89 & 5.4e-5$|$5.6e-3 & 4.7e-3$|$4.7e-3 & 10$|$10 \\
		3000 &99.9\%& 50.19$|$73.35 & 11.33$|$32.81 & 26.05$|$25.06 & 1.13$|$3.28 & 2.61$|$2.51 & 2.8e-5$|$4.6e-3 & 3.9e-3$|$3.9e-3 & 10$|$10 \\
		4000 &99.9\%& 72.62$|$114.68 & 18.56$|$47.25 & 33.92$|$40.54 & 1.86$|$4.73 & 3.39$|$4.05 & 1.8e-5$|$4.0e-3 & 3.4e-3$|$3.4e-3 & 10$|$10 \\
		5000 &99.9\%& 119.99$|$189.62 & 31.53$|$81.12 & 55.05$|$64.73 & 3.15$|$8.11 & 5.50$|$6.47 & 1.3e-5$|$3.6e-3 & 3.0e-3$|$3.0e-3 & 10$|$10 \\ \hline
				\end{tabular}}
			\end{center}}
		\end{table}
\begin{table}[tbhp]
	{\footnotesize
		\caption{The main time cost on SNL by \texttt{MPA} and \texttt{SQREDM} with $R=50$} \label{tab:SNL_time_R50}
		\begin{center}
			\resizebox{\textwidth}{11mm}{
				\begin{tabular}{|c|c|c|c|c|c|c|c|c|c|} \hline
					\multirow{2}{*}{n}&\multirow{2}{*}{$\Delta_{rate}$}& t(s) &$t_{sub}(s)$&$t_{eig}(s)$&$Ave_{sub}$&$Ave_{eig}$& RMSD & rRMSD & Iter\\ \cline{3-10}
					~  &~  & A1$|$A3 & A1$|$A3 & A1$|$A3 & A1$|$A3& A1$|$A3& A1$|$A3& A1$|$A3& A1$|$A3 \\ \hline
					1000 & 48.5\% & 10.72$|$12.56 & 0.88$|$2.58 & 0.53$|$0.55 & 0.09$|$0.26 & 0.05$|$0.05 & 1.5e-3$|$7.7e-2 & 9.7e-3$|$9.7e-3 & 10$|$10 \\ 
					1500 & 48.1\% & 33.02$|$36.51 & 2.05$|$5.48 & 1.21$|$1.23 & 0.20$|$0.55 & 0.12$|$0.12 & 1.4e-3$|$9.2e-2 & 7.9e-3$|$7.9e-3 & 10$|$10 \\ 
					2000 & 48.2\% & 73.92$|$79.91 & 3.72$|$9.38 & 2.08$|$2.14 & 0.37$|$0.94 & 0.21$|$0.21 & 1.4e-3$|$1.0e-1 & 6.8e-3$|$6.8e-3 & 10$|$10 \\ 
					3000 & 48.4\% & 238.90$|$250.80 & 8.76$|$20.78 & 4.82$|$4.80 & 0.88$|$2.08 & 0.48$|$0.48 & 1.5e-3$|$1.3e-1 & 5.6e-3$|$5.6e-3 & 10$|$10 \\ 
					4000 & 48.3\% & 598.16$|$597.84 & 20.99$|$39.40 & 20.52$|$20.20 & 2.10$|$3.94 & 2.05$|$2.02 & 1.5e-3$|$1.5e-1 & 4.8e-3$|$4.8e-3 & 10$|$10 \\ 
					5000 & 48.6\% & 1191.88$|$1246.91 & 47.93$|$99.07 & 70.45$|$87.05 & 4.79$|$9.91 & 7.04$|$8.70 & 1.6e-3$|$1.7e-1 & 4.3e-3$|$4.3e-3 & 10$|$10 \\ \hline
				\end{tabular}}
			\end{center}}
		\end{table}
\begin{figure}[htbp]
	\centering
	\includegraphics[height=6cm,width=7cm]{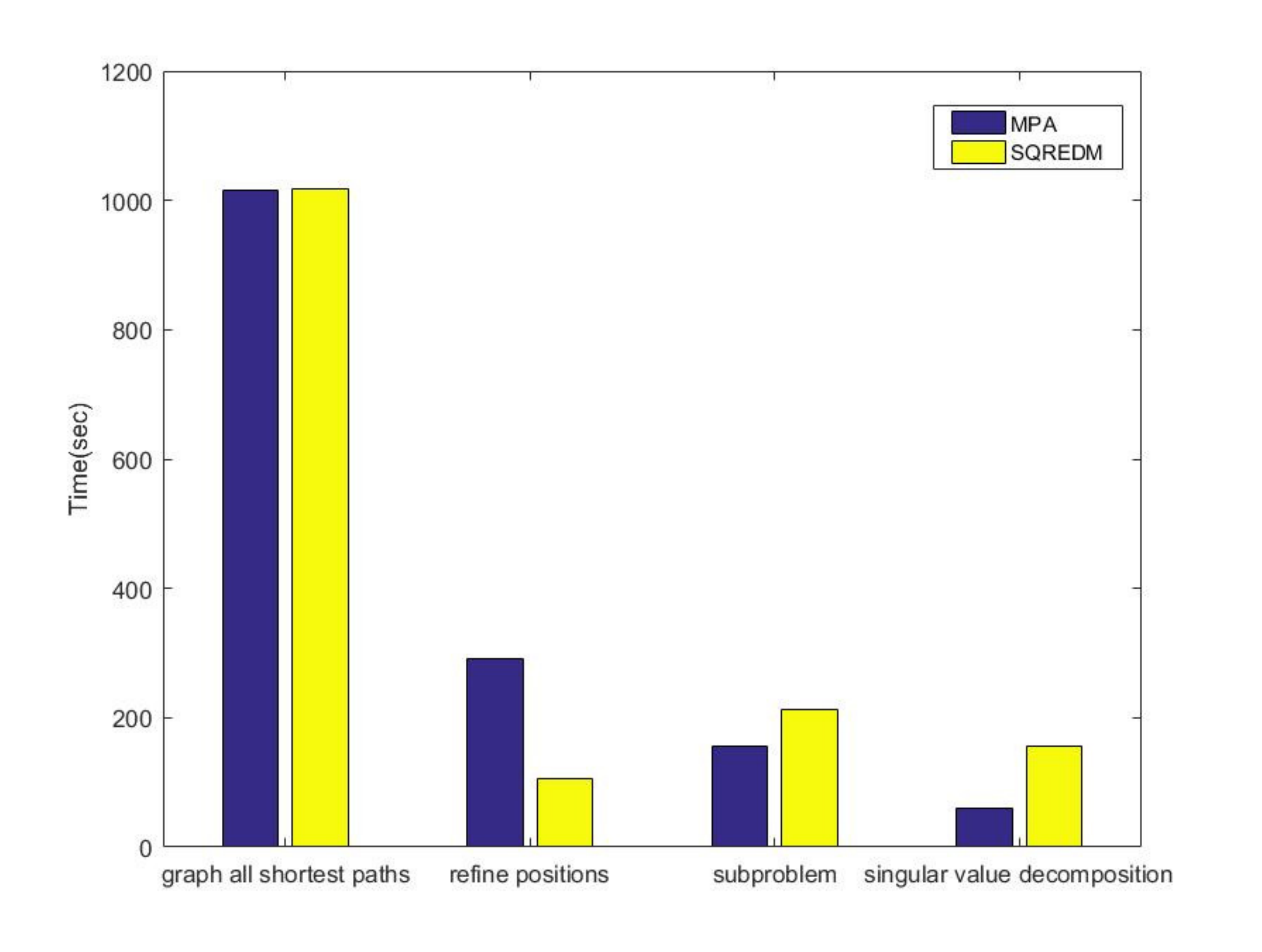}
	\caption{Comparison of \texttt{MPA} and \texttt{SQREDM} with $n=5000$ and $R=50$}	
	\label{fig:histogram}
\end{figure}		
		
{\bf{Molecular Conformation.}} Molecular conformation has long been an important application of EDM optimization \cite{Glunt2010Molecular}. These problems represent a very challenging set of embedding problems in three dimensions (r = 3). We collect real data of 20 molecules derived from 12 structures of proteins from the Protein Data Bank \cite{Helen2000The}. We generate $\Delta$ following the way as in \cite{Zhou2017A}, and take $s=r=3$. 

Similar to SNL test, we denote $\mathscr{N}_x$ as the set formed by indices of measured distances. If $\delta_{ij}>0$, let $W_{ij}=1$. Otherwise, $W_{ij}=0$. The noise factor $nf=0.1$. The ordinal constraints are generated in the same way as in SNL.
We compare \texttt{MPA} with \texttt{SQREDM}. For \texttt{SQREDM}, all parameters are {set as} default. The results are reported in Table \ref{Tab:MC}. It can be observed that \texttt{MPA} {performs} better {in terms of} RMSD and rRMSD. This is reasonable since our model includes ordinal information while \texttt{SQREDM} solves EDM model (\ref{prob-qi}) with box constraints. As for cputime, \texttt{MPA} is slightly faster than \texttt{SQREDM}. Typical results are demonstrated in Fig. \ref{fig:MC_EDMOC}, \ref{fig:MC_SQREDM}.

\begin{table}[tbhp]
	{\footnotesize
		\caption{Results on MC by \texttt{MPA} and \texttt{SQREDM}} \label{Tab:MC}
		\begin{center}
				\begin{tabular}{|c|c|c|c|c|c|c|} \hline
					
					\multirow{2}{*}{Protein} & \multirow{2}{*}{n} & \multirow{2}{*}{m} & \multirow{2}{*}{$\Delta_{rate}$} & t(s) & RMSD & rRMSD \\ \cline{5-7}
					~ &~ &~&~  & A1$|$A3 & A1$|$A3 & A1$|$A3 \\ \hline
		1PBM & 126 & 7875 &12.5\%& 0.12$|$0.17 & 6.63e-2$|$1.60 & 1.58e-1$|$3.44e-1 \\
		5BNA & 243 & 29403 &6.2\%& 0.19$|$0.25 & 9.67e-1$|$3.33 & 6.22e-1$|$2.53 \\
		1PTQ & 402 & 80601 &4.4\%& 0.34$|$0.58 & 4.59e-3$|$8.28e-1 & 1.53e-1$|$2.75e-1 \\
		1LFB & 641 & 205120 &2.8\%& 0.72$|$1.15 & 2.11e-2$|$1.39 & 1.54e-1$|$4.74e-1 \\
		1PHT & 666 & 221445 &2.8\%& 0.81$|$1.22 & 8.83e-2$|$1.76 & 1.45e-1$|$1.13 \\
		1DCH & 806 & 324415 &2.4\%& 1.13$|$1.70 & 2.08e-2$|$1.02 & 1.45e-1$|$2.08e-1 \\
		1HQQ & 891 & 396495 &2.1\%& 1.28$|$1.95 & 5.32e-3$|$1.47 & 1.48e-1$|$5.65e-1 \\
		1POA & 914 & 417241 &2.0\%& 1.30$|$2.05 & 3.15e-2$|$1.42 & 1.36e-1$|$3.65e-1 \\
		1RHJ & 1113 & 618828 &1.5\%& 2.06$|$2.72 & 1.81e-1$|$3.84 & 1.56e-1$|$3.42 \\
		1TJO & 1394 & 970921 &1.3\%& 2.49$|$4.08 & 9.09e-2$|$3.02 & 1.41e-1$|$2.71 \\
		1TIM & 1870 & 1747515 &1.0\%& 4.39$|$6.99 & 1.92e-2$|$1.16 & 1.30e-1$|$3.46e-1 \\
		1RGS & 2015 & 2029105 &0.9\%& 5.04$|$8.19 & 1.07e-2$|$1.87 & 1.23e-1$|$6.09e-1 \\
		1TOA & 2147 & 2303731 &0.9\%& 6.03$|$9.58 & 2.00e-2$|$1.19 & 1.28e-1$|$3.44e-1 \\
		1NFB & 2833 & 4011528 &0.6\%& 11.09$|$15.65 & 8.08e-2$|$3.53 & 1.49e-1$|$2.77 \\
		1KDH & 2846 & 4800351 &0.7\%& 14.24$|$16.21 & 4.22e-2$|$2.97 & 1.19e-1$|$1.28 \\
		1PBB & 3099 & 5915080 &0.6\%& 11.97$|$18.57 & 4.37e-2$|$1.69 & 1.20e-1$|$4.78e-1 \\
		1NF7 & 3440 & 6126750 &0.5\%& 14.83$|$23.41 & 2.67e-2$|$5.23 & 1.28e-1$|$4.57 \\
		1NFG & 3501 & 16134040 &0.6\%& 15.90$|$23.98 & 1.66e-2$|$1.06 & 1.12e-1$|$2.95e-1 \\
		1QRB & 4119 & 8481021 &0.5\%& 30.54$|$38.94 & 1.64e-1$|$7.21 & 1.85e-1$|$6.88 \\
		1MQQ & 5681 & 16134040 &0.4\%& 76.20$|$99.32 & 2.27e-2$|$1.32 & 1.10e-1$|$3.36e-1 \\  \hline
				\end{tabular}
			\end{center}
		}
	\end{table}

\begin{figure}[htbp]	
	\centering
	\includegraphics[height=5.5cm,width=12.8cm]{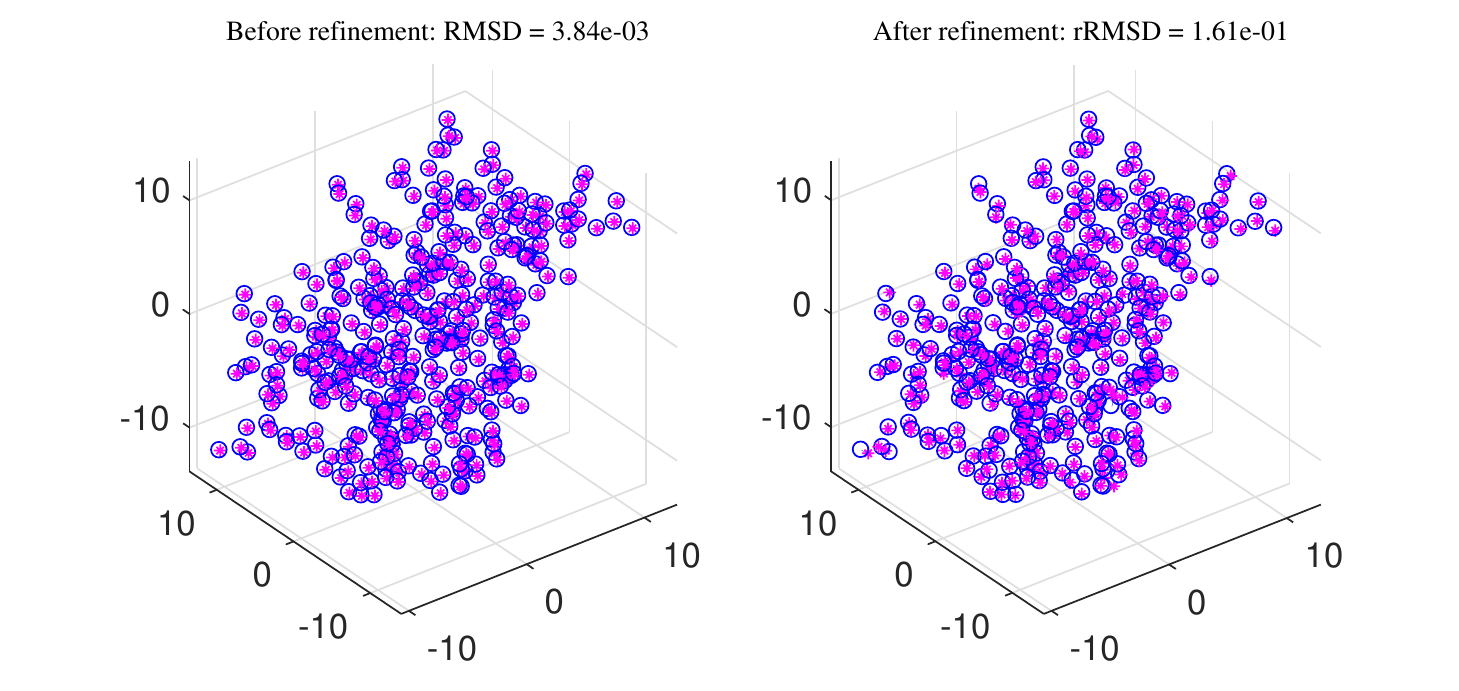}
	\caption{1PTQ by \texttt{MPA}}
	\label{fig:MC_EDMOC}
\end{figure}
\begin{figure}[htbp]	
	\centering
	\includegraphics[height=5.5cm,width=12.8cm]{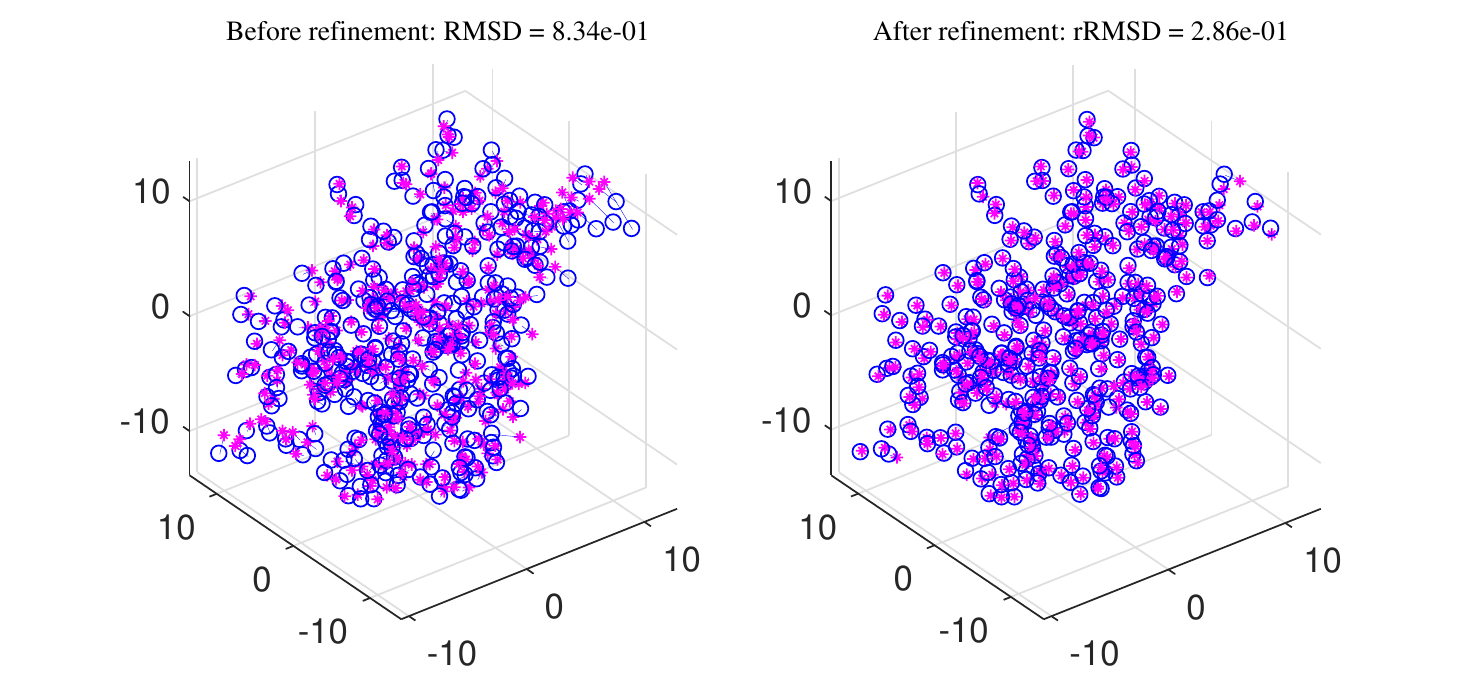}
	\caption{1PTQ by \texttt{SQREDM}}
	\label{fig:MC_SQREDM}
\end{figure}

\section{Conclusions}\label{sec-6}
	In this paper, we studied the Euclidean distance matrix optimization with ordinal constraints, which is of great importance in both theory and application. We investigated the feasibility of EDMOC systematically. As far as we know, this is the first time that the feasibility of EDMOC has been investigated. We showed that a nonzero solution always exists for EDMOC without the rank constraint. For the full ordinal constraints case, we showed that a nontrivial solution exists for $r \ge n-2$. An example was given for $r<n-2$ showing that EDMOC may only admit zero feasible solution.
We developed a majorized penalty method to solve large-scale EDMOC model and convex EDM optimization problem into an isotonic regression problem. Due to the partial spectral decomposition and the fast solver for isotonic regression problem, the performance of the proposed algorithm has been demonstrated to be superior both in terms of the solution quality and cputime in sensor network localization and molecular conformation.

 {Note that  the feasibility of EDMOC in the case of $r<n-2$ is only partly answered}. Given the fact that full ordinal constraints may result in only zero solutions, it brings another deep and challenging question: to make EDMOC have a nonzero solution, how to select proper ordinal constraints?
These questions will be investigated in future.

\appendix
\

\section{Readmap of Proof}
Due to the fact that $\mathcal{K}_+^n(n-2)\subset \mathcal{K}_+^n(n-1)=\mathcal{K}_+^n$, we only need to show the result holds for $r=n-2$.
For any {ordinal constraints}, we consider the situation when $r=n-2$. We {will} construct $n$ points in $\Re^{n-2}$ satisfying that the largest pairwise distance is strictly larger than the others and the rest are equal to each other. By relabelling the vertices properly, the resulting EDM can satisfy the required ordinal constraints.

To construct the points as required, inspired by { the regular simplex} whose edges are equal {to each other}, { we start from a regular simplex in $\Re^{n-3}$, and} add two extra vertices in a higher dimensional space such that each of the extra vertices together with the original regular simplex form a regular simplex in a higher dimensional space {$\Re^{n-2}$}. In this way, the distance between the two additional vertices is proved to be larger than the rest distances. It results in a polyhedron which is formed by two regular simplices in $\Re^{n-1}$ who share $n-2$ vertices. The $n$ vertices of the polyhedron are the points that we are looking for. Fig. \ref{triangleO}, \ref{tetrahedron}, \ref{All the five points} demonstrate the process when $n=5$.

Before we give the details of the proof of Theorem \ref{thm-3}, we start with the following lemmas.
\begin{lemma}\label{lem-7-1}
		Let $\{b_k\}$ {be} a nonnegative sequence generated as follows
	\be\label{lemma1-b}	
	\left\{
	\begin{aligned}
		b_1 \ \ = \ \ \frac{t}{2} \\
		b_k^2 +(\sqrt{t^2-b_k^2}-b_{k+1})^2& = & b_{k+1}^2, \ \ k=1,2\dots,
	\end{aligned}
	\right.	
	\ee
	where $t>0$.
	Then $\{b_k\}$ is an increasing sequence and converges to $\frac{\sqrt{2}}{2}t$.
	
	\begin{proof}[Proof of Lemma \ref{lem-7-1}] The update in ({\ref{lemma1-b}}) implies that
		\begin{displaymath}
		0\le b_k \le t,\, \,k=1,2\dots.
		\end{displaymath}
		{The second equation in (\ref{lemma1-b}) also gives the increasing order of $\{b_k\}$, i.e.,}
		\begin{displaymath}
		b_1<b_2<\dots<b_k<\dots.
		\end{displaymath}
		In other words, $\{b_k\}$ is an increasing sequence with upper bound $t$. Therefore, $\{b_k\}$ has a limit as $k\to \infty$. Suppose the limit is $b$. By taking limit in the second equation in (\ref{lemma1-b}), we can get \[t^2=2b\sqrt{t^2-b^2}.\] This gives {the} solution $b=\frac{\sqrt{2}}{2}$t. The proof is finished.
	\end{proof}
	{\bf{Remark.}} Equations in (\ref{lemma1-b}) imply that for all $k=1,2,\dots$, $$b_k<\frac{\sqrt{2}}{2}t.$$
\end{lemma}

\begin{lemma}\label{lemma2}
	Let ${x_1^{(k)},\dots,x_{k+1}^{(k)}}\subseteq \Re^k$ be a set of points satisfying \be\label{equ3}\|x_i^{(k)}-x_j^{(k)}\|=t,\, i\ne j,\ee
where the {superscript} stands for the dimension of vectors. Denote the centroid of $\{x_1^{(k)},\dots, x_{k+1}^{(k)}\}$ as $O^{(k)}$.
	Let the distance between $O^{(k)}$ and $x_i^{(k)}$ {be} $b_k$. We have the following equations
	$$\left\{
	\begin{array}{ll}
	b_1 = \frac{t}{2},& \\
	b_k^2 +(\sqrt{t^2-b_k^2}-b_{k+1})^2 = b_{k+1}^2,& \ k=1,2\dots.
	\end{array}
	\right.
	$$
\end{lemma}

\begin{proof}[Proof of Lemma \ref{lemma2}] Without the loss of generality, let $O^{(k)}$ lie in the {origin}. Then $b_k=\|O^{(k)}-x_i^{(k)}\|, \ \ i=1,\dots,k+1$. Now define ${x_1^{(k+1)},\dots,x_{k+2}^{(k+1)}}\subseteq \Re^{k+1}$ as follows
	$$x_i^{(k+1)}=
	\begin{bmatrix}
	x_i^{(k)} \\
	0
	\end{bmatrix}, \ \ i=1,\dots,k+1,\ \ \
{x^{(k+1)}_{k+2}}=
	\begin{bmatrix}
	{O^{(k)}} \\
	{h_{k+1}}
	\end{bmatrix}, \ \ h_{k+1}>0.$$
	By letting $h_{k+1}=\sqrt{t^2-b_k^2}$, there is
	\be\label{equ4}\|x_i^{(k+1)}-x_j^{(k+1)}\|=t,\,\ \forall\ \ i<j, \ \ i,\ j=1,\dots,k+2.\ee
 Let the centroid of ${x_1^{(k+1)},\dots,x_{k+2}^{(k+1)}}$ be $O^{(k+1)}$. {There is $$O^{(k+1)}=\frac{1}{k+2}\sum_{i=1}^{k+2}x^{(k+1)}_i.$$
		Notice \[{\hat{O}^{(k)}}:=\begin{bmatrix}
	O^{(k)}\\
	0
	\end{bmatrix}
	=\frac{1}{k+1}\sum_{i=1}^{k+1}x^{(k+1)}_i.
	\]We get  $O^{(k+1)}=\frac{k+1}{k+2}\hat{O}^{(k)}+\frac{1}{k+2}x_{k+2}^{(k+1)}$ which implies that $O^{(k+1)}$ lies between $\hat{O}^{(k)}$ and $x_{k+2}^{(k+1)}$.} The geometric relation of $O^{(k+1)},\hat{O}^{(k)},x_i^{(k+1)},x_{k+2}^{(k+1)}$ is {demonstrated} in Fig. \ref{fig-thm-7}.
\begin{figure}[htbp]
		\centering
		\includegraphics[height=5.5cm,width=4.3cm]{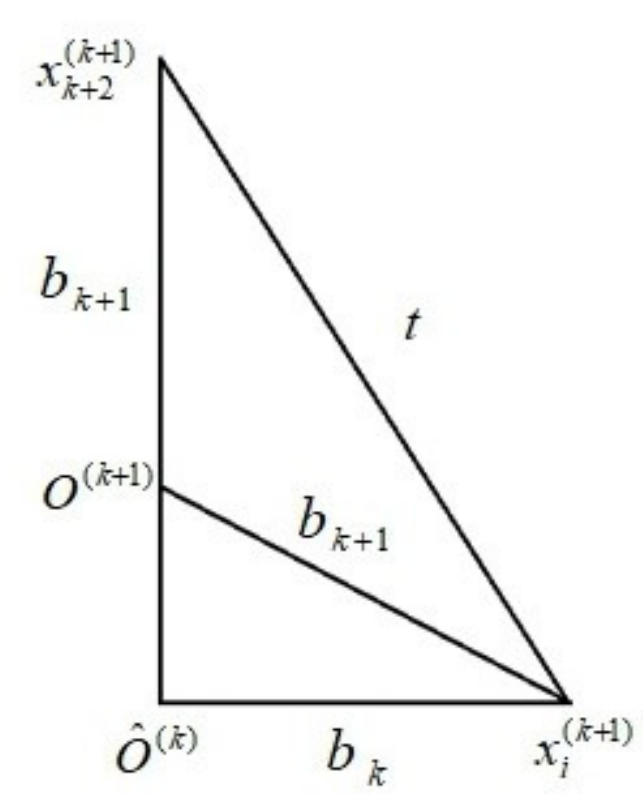}
		\caption{Geometric relation of $O^{(k+1)},\ \hat{O}^{(k)},\ x_i^{(k+1)},\ x_{k+2}^{(k+1)}$}
		\label{fig-thm-7}
\end{figure}	
	
	 Consequently, we have the following relationship $$t^2=b_k^2+(\sqrt{t^2-b_k^2}-b_{k+1})^2.$$ In particular, if k=1, $b_1=\frac{t}{2}$. The proof is finished.
\end{proof}

Next we give the proof of Theorem \ref{thm-3}.
\begin{proof}[\bf Proof of Theorem \ref{thm-3}]Note that $\mathcal{K}_+^n(n-2)\subset \mathcal{K}_+^n(n-1)=\mathcal{K}_+^n$, we only need to show the result holds for $r=n-2$.

Let $r=n-2$. To show the result, let's first construct a nontrivial feasible solution $D$ satisfying
\be\label{equ6}
\left\{
\begin{array}{ll}
D_{12}>D_{ij},&\ \ \forall\ (i,j)\ne(1,2),\ \,\ i<j, \\
D_{ij}=D_{sk},&\ \ \forall\ (i,j),\ (s,k)\ne(1,2),\, i<j,\ \, s<k.
\end{array}
\right.
\ee
To the end, we can first pick up a set of points $\{ x_1^{(n-3)},\dots,x_{n-2}^{(n-3)} \}\subseteq \Re^{n-3}$ satisfying (See Fig. \ref{triangleO} for $n=5$) \[\|x_i^{(n-3)}-x_j^{(n-3)}\|=t,\ \ \forall\ \ i<j,\, i,\ j=1,\dots,{n-2}.\] This is already proved in Theorem \ref{thm-1}. Moreover, denote the centroid of $x_1^{(n-3)},$ $\dots,$ $
x_{n-2}^{(n-3)}$ as $O^{(n-3)}$. Suppose $O^{(n-3)}$ is located at the origin. As we defined above, let \[b_{n-3}=\|O^{(n-3)}-x_i^{(n-3)}\|, \, \forall\ i=1,\dots,n-2.\] Now let $x_i^{(n-2)}\in\Re^{n-2},\ i=1,\dots,n-1$ be generated as follows (See Fig. \ref{tetrahedron} and Fig. \ref{All the five points} for $n=5$)

\begin{equation*}
x_i^{(n-2)}=
\begin{bmatrix}
x_i^{(n-3)} \\
0
\end{bmatrix}, \ \ i=1,\dots,n-2,\ \ x_{n-1}^{(n-2)}=
\begin{bmatrix}
{O^{(n-3)}} \\
h
\end{bmatrix}
, \ \ x_n^{(n-2)}=
\begin{bmatrix}
{O^{(n-3)}} \\
-h
\end{bmatrix} ,\\ h>0.
\end{equation*}
Then we have $$(\hat{O}^{(n-3)}-x_{n-1}^{(n-2)})^T(x_i^{(n-2)}-x_j^{(n-2)})=0,\ \ \forall \, i<j,\ \ i,\ j=1,\dots,{n-2},$$ where \[{\hat{O}^{(n-3)}=\begin{bmatrix}
O^{(n-3)}\\
0
\end{bmatrix}
\in \Re^{n-2}.}
\]
By letting $h=\sqrt{t^2-b_{n-3}^2}$, we get $$\|x_i^{(n-2)}-x_j^{(n-2)}\|=t,\ \, \forall\ i<j, \ \ i,\ j=1,\dots,n,\ \ (i,j)\ne{(n-1,n)}.$$
Next, we will show that \[\|x_{n-1}^{(n-2)}-x_{n}^{(n-2)}\|>t.\]
By Lemma \ref{lemma2}, we have (\ref{equ3}) for $b_k$ and $b_{k+1}$, implying that $b_{n-3}<\frac{\sqrt{2}t}{2}$ {by Lemma \ref{lem-7-1}}. Therefore, $h=\sqrt{t^2-b_{n-3}^2}>\frac{\sqrt{2}t}{2}$ for all $n\ge4$. In other words,
\[\|x_n^{(n-2)}-x_{n-1}^{(n-2)}\|=2h>\sqrt{2}t>t.\]
Consequently, the EDM generated by $\{x_1^{(n-2)},\dots,x_n^{(n-2)}\}\subset\Re^{n-2}$ satisfies (\ref{equ6}).
For any $\pi(n)=\{(i_1,j_1),\dots,(i_m,j_m)\}$, to get a nontrivial solution of $F(\pi(n),n-2)$, we construct $$\hat{D}_{i_1j_1}=\|x_n^{(n-2)}-x_{n-1}^{(n-2)}\|^2, \quad \hat{D}_{ij}=\|x_i^{(n-2)}-x_j^{(n-2)}\|^2, \quad (i,j)\in\pi(n)\backslash (i_1,j_1).$$
Then $\hat{D}$ is a nontrivial feasible solution.
\end{proof}

\begin{figure}[htbp]
	\begin{minipage}[t]{0.3\linewidth}
		\centering
		\includegraphics[height=3cm,width=3.75cm]{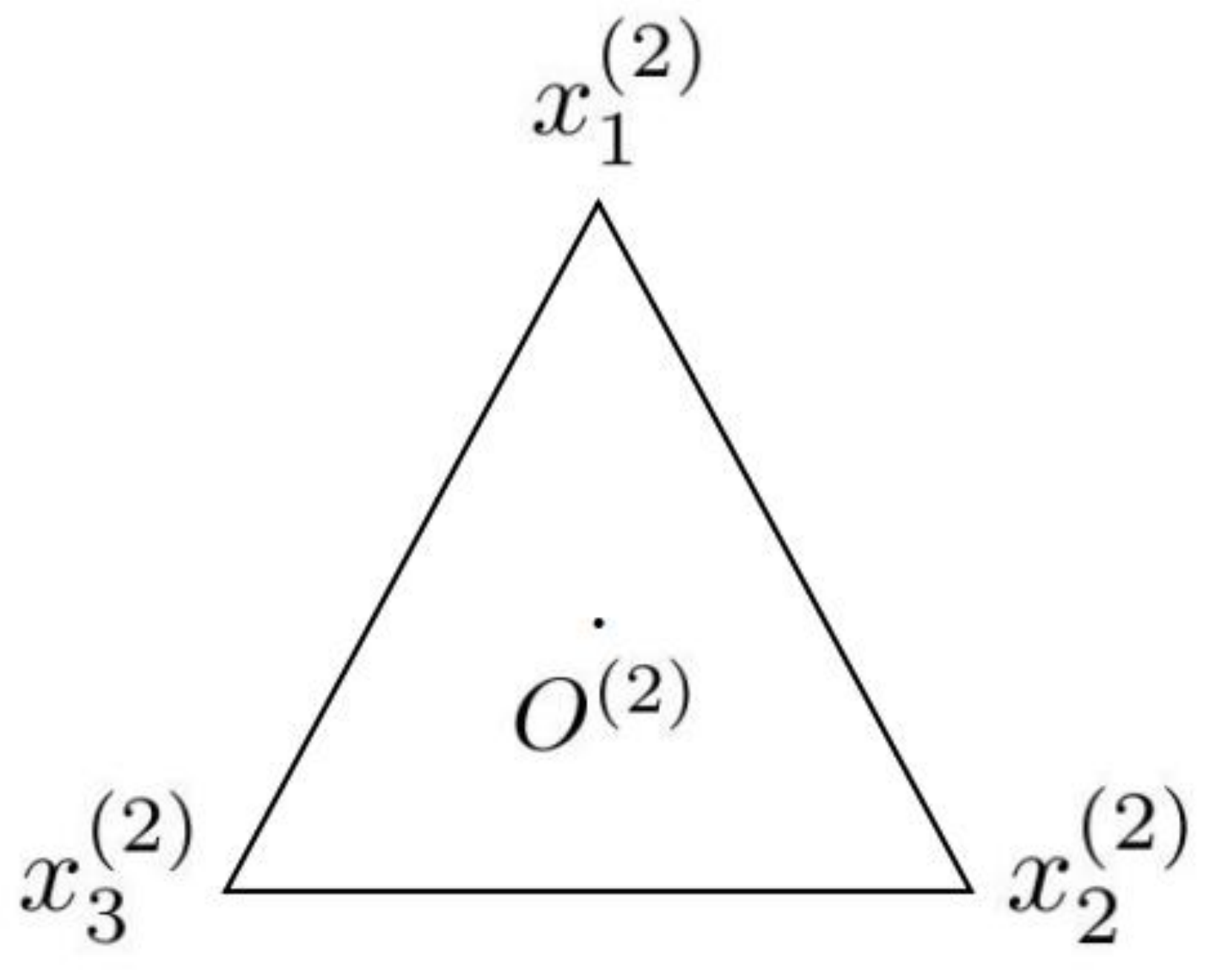}
		\caption{$x_1^{(2)},x_2^{(2)},x_3^{(2)}$}
		\label{triangleO}
	\end{minipage}
	\begin{minipage}[t]{0.3\linewidth}
		\centering
		\includegraphics[height=3cm,width=3.75cm]{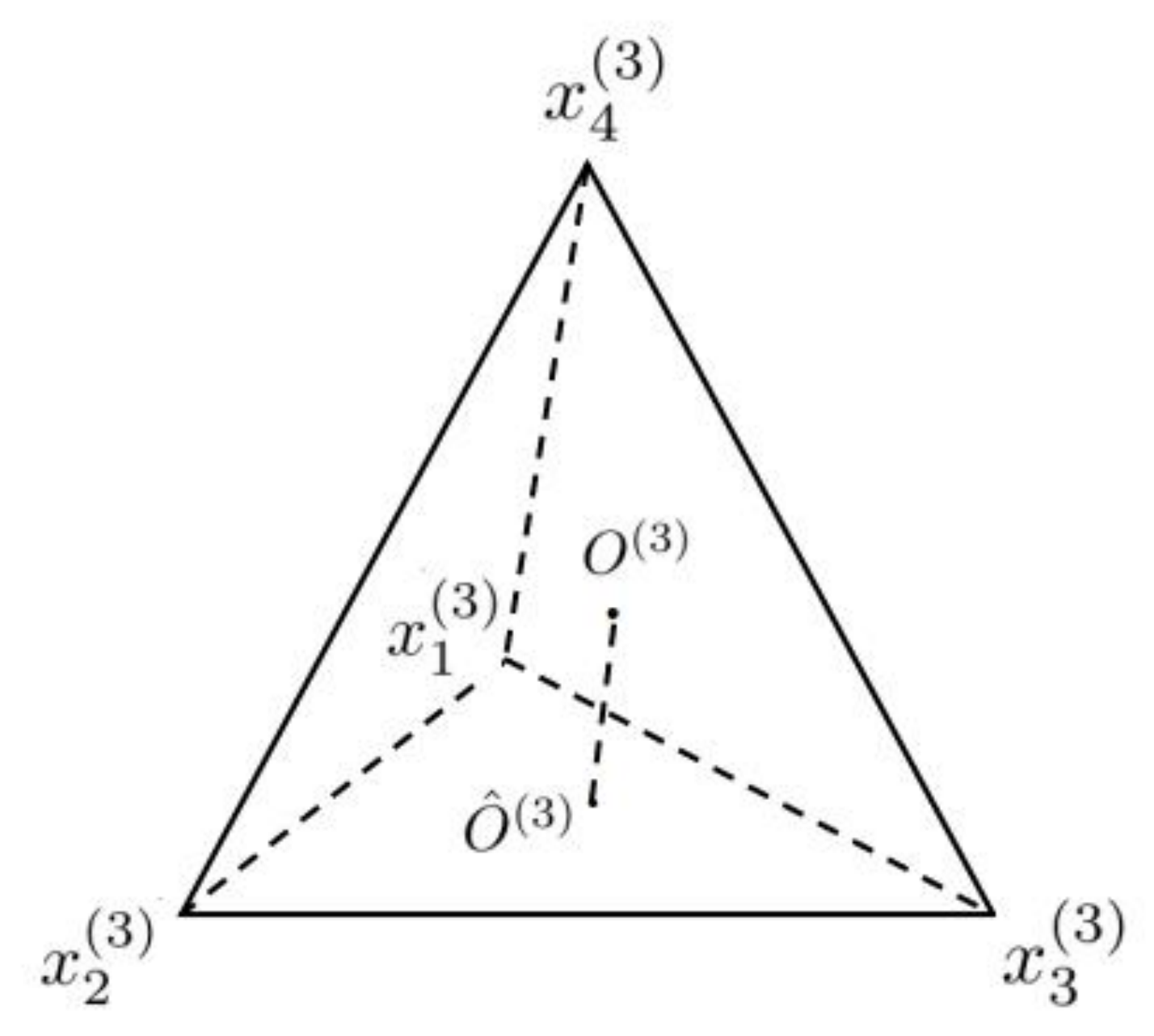}
		\caption{$x_1^{(3)},\dots, x_4^{(3)}$}
		\label{tetrahedron}
	\end{minipage}
	\begin{minipage}[t]{0.3\linewidth}
		\centering
		\includegraphics[height=3.5cm,width=3.345cm]{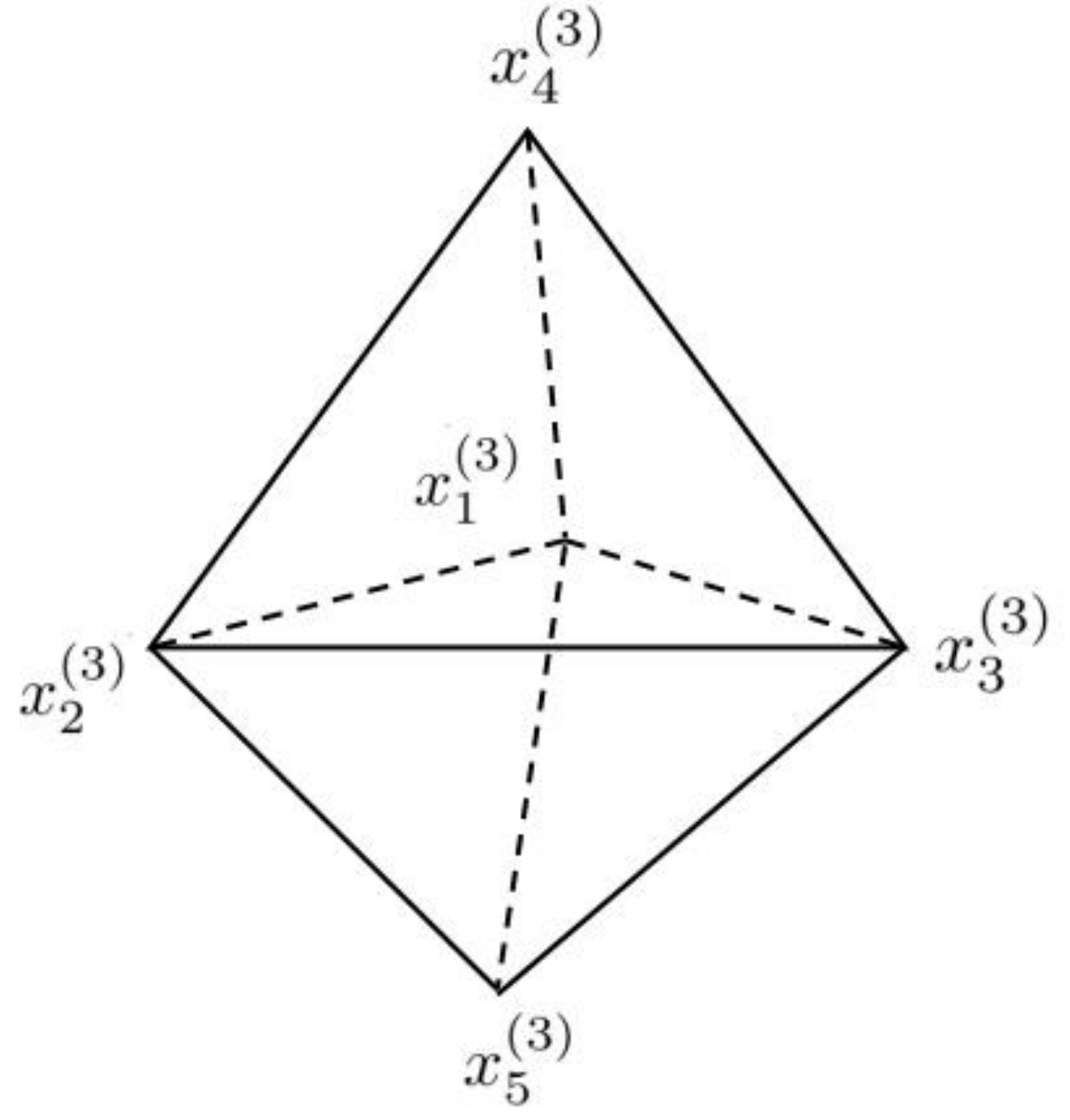}
		\caption{$x_1^{(3)},\dots, x_5^{(3)}$}
		\label{All the five points}
	\end{minipage}
\end{figure}


\begin{thebibliography}{10}

\bibitem{Bai2015Tackling}
S.~H. Bai and H.~D. Qi.
\newblock Tackling the flip ambiguity in wireless sensor network localization
  and beyond.
\newblock {\em Digital Signal Processing}, 55(C):85--97, 2016.

\bibitem{barlow1972statistical}
R.~E. Barlow, Bartholomew~D. J., J.~M. Bremner, and H.~D. Brunk.
\newblock {\em Statistical Inference under Order Restrictions: The Theory and
  Application of Isotonic Regression}.
\newblock Wiley, 1973.

\bibitem{Helen2000The}
H.~M. Berman, J.~Westbrook, Z.~K. Feng, G.~Gilliland, T.~N. Bhat, H.~Weissig,
  I.~N. Shindyalov, and P.~E. Bourne.
\newblock The protein data bank.
\newblock {\em Nucleic Acids Research}, 28(1):235--242, 2000.

\bibitem{Biswas2006Semidefinite}
P.~Biswas, T.~C. Liang, K.~C. Toh, Y.~Ye, and T.~C. Wang.
\newblock Semidefinite programming approaches for sensor network localization
  with noisy distance measurements.
\newblock {\em IEEE Transactions on Automation Science and Engineering},
  3(4):360--371, 2006.

\bibitem{biswas2004semidefinite}
P.~Biswas and Y.~Y. Ye.
\newblock Semidefinite programming for ad hoc wireless sensor network
  localization.
\newblock In {\em Proceedings of the 3rd International Symposium on Information
  Processing in Sensor Networks}, pages 46--54. 2004.

\bibitem{Bogdan2015SLOPE}
M.~Bogdan, D.~B.~E. Van, C.~Sabatti, W.~Su, and E.~J. Candès.
\newblock \hbox{SLOPE}-adaptive variable selection via convex optimization.
\newblock {\em Annals of Applied Statistics}, 9(3):1103--1140, 2015.

\bibitem{Borg2010Modern}
I.~Borg and P.~Groenen.
\newblock Modern multidimensional scaling: Theory and applications.
\newblock {\em Journal of Educational Measurement}, 40(3):277--280, 2010.

\bibitem{Borg}
I.~Borg and P.~J.~F. Groenen.
\newblock {\em Modern Multidensional Scaling}.
\newblock Springer, 2005.

\bibitem{Cox}
T.~F. Cox and M.~A.~A. Cox.
\newblock {\em Multidimensional Scaling}.
\newblock Chapman and Hall/CRC, 2000.

\bibitem{Dattorro}
J.~Dattorro.
\newblock {\em Convex Optimization and \hbox{E}uclidean Distance Geometry}.
\newblock Meboo, 2005.

\bibitem{deLeeuw1977}
J.~De~Leeuw.
\newblock Applications of convex analysis to multidimensional scaling.
\newblock {\em Recent Developments in Statistics}, pages 133--146, 2011.

\bibitem{deLeeuw2009}
J.~De~Leeuw and P.~Mair.
\newblock Multidimensional scaling using majorization: \hbox{SMACOF} in
  \hbox{R}.
\newblock {\em Journal of Statistical Software}, 31(3):1--30, 2009.

\bibitem{QiDing2015}
C.~Ding and H.~D. Qi.
\newblock Convex \hbox{E}uclidean distance embedding for collaborative position
  localization with \hbox{NLOS} mitigation.
\newblock {\em Computational Optimization and Applications}, 66(1):187--218,
  2017.

\bibitem{elte1912semiregular}
E.~L. Elte.
\newblock {\em The Semiregular Polytopes of the Hyperspaces}.
\newblock Hoitsema, 1912.

\bibitem{Fang2013Using}
X.~Y. Fang and K.~C. Toh.
\newblock Using a distributed \hbox{SDP} approach to solve simulated protein
  molecular conformation problems.
\newblock In {\em Distance Geometry}, pages 351--376. Springer, 2013.

\bibitem{gao2009calibrating}
Y.~Gao and D.~F. Sun.
\newblock Calibrating least squares covariance matrix problems with equality
  and inequality constraints.
\newblock {\em SIAM Journal on Matrix Analysis}, 31(3):1432--1457, 2009.

\bibitem{Glunt2010Molecular}
W.~Glunt, T.~L. Hayden, and M.~Raydan.
\newblock Molecular conformations from distance matrices.
\newblock {\em Journal of Computational Chemistry}, 14(1):114--120, 1993.

\bibitem{Gower}
J.~C. Gower.
\newblock Properties of \hbox{E}uclidean and non-\hbox{E}uclidean distance
  matrices.
\newblock {\em Linear Algebra and Its Applications}, 67(none):81--97, 1985.

\bibitem{Kruskal1964Multidimensional}
J.~B. Kruskal.
\newblock Multidimensional scaling by optimizing goodness of fit to a nonmetric
  hypothesis.
\newblock {\em Psychometrika}, 29(1):1--27, 1964.

\bibitem{Kruskal1964Nonmetric}
J.~B. Kruskal.
\newblock Nonmetric multidimensional scaling: \hbox{A} numerical method.
\newblock {\em Psychometrika}, 29(2):115--129, 1964.

\bibitem{DISCO}
N.~Leung, Z.~Hang, and K.~C. Toh.
\newblock An \hbox{SDP}-based divide-and-conquer algorithm for large-scale
  noisy anchor-free graph realization.
\newblock {\em SIAM Journal on Scientific Computing}, 31(6):4351--4372, 2009.

\bibitem{LiQi2017}
Q.~N. Li and H.~D. Qi.
\newblock An inexact smoothing \hbox{N}ewton method for \hbox{E}uclidean
  distance matrix optimization under ordinal constraints.
\newblock {\em Journal of Computational Mathematics}, 35(4):469--485, 2017.

\bibitem{Liberti}
L.~Liberti, C.~Lavor, N.~Maculan, and A.~Mucherino.
\newblock \hbox{E}uclidean distance geometry and applications.
\newblock {\em Quantitative Biology}, 56(1):3--69, 2012.

\bibitem{Mordukhovich2006Variational}
B.~S. Mordukhovich.
\newblock {\em Variational Analysis and Generalized Differentiation \hbox{I}}.
\newblock Springer, 2006.

\bibitem{Qi2013}
H.~D. Qi.
\newblock A semismooth \hbox{N}ewton's method for the nearest \hbox{E}uclidean
  distance matrix problem.
\newblock {\em SIAM Journal on Matrix Analysis and Applications},
  34(34):67--93, 2013.

\bibitem{Qi2014}
H.~D. Qi.
\newblock Conditional quadratic semidefinite programming: \hbox{E}xamples and
  methods.
\newblock {\em Journal of the Operations Research Society of China},
  2(2):143--170, 2014.

\bibitem{QiXiuYuan2013}
H.~D. Qi, N.~H. Xiu, and X.~M. Yuan.
\newblock A \hbox{L}agrangian dual approach to the single-source localization
  problem.
\newblock {\em IEEE Transactions on Signal Processing}, 61(15):3815--3826,
  2013.

\bibitem{Qi2014Computing}
H.~D. Qi and X.~M. Yuan.
\newblock Computing the nearest \hbox{E}uclidean distance matrix with low
  embedding dimensions.
\newblock {\em Mathematical Programming}, 147(1-2):351--389, 2014.

\bibitem{SMACOF}
G.~Rosman, A.~M. Bronstein, M.~M. Bronstein, A.~Sidi, and R.~Kimmel.
\newblock Fast multidimensional scaling using vector extrapolation.
\newblock Technical report, Computer Science Department, Technion, 2008.

\bibitem{Schoenberg1935Remarks}
I.~J. Schoenberg.
\newblock Remarks to maurice frechet's article ``sur la definition axiomatique
  d'une classe d'espace distances vectoriellement applicable sur l'espace de
  hilbert.
\newblock {\em Annals of Mathematics}, 36(3):724--732, 1935.

\bibitem{Toh2008}
K.~C. Toh.
\newblock An inexact primal-dual path-following algorithm for convex quadratic
  \hbox{SDP}.
\newblock {\em Mathematical Programming}, 112(1):221--254, 2008.

\bibitem{Torgerson1952Multidimensional}
W.~S. Torgerson.
\newblock Multidimensional scaling: \hbox{I}. theory and method.
\newblock {\em Psychometrika}, 17(4):401--419, 1952.

\bibitem{YoungHouseholder}
G.~Young and A.~S. Householder.
\newblock Discussion of a set of points in terms of their mutual distances.
\newblock {\em Psychometrika}, 3(1):19--22, 1938.

\bibitem{ZhaiLi2019}
F.~Z. Zhai and Q.~N. Li.
\newblock A \hbox{E}uclidean distance matrix model for protein molecular
  conformation.
\newblock {\em Journal of Global Optimization}, 2019.

\bibitem{Zhou2017A}
S.~L. Zhou, N.~H. Xiu, and H.~D. Qi.
\newblock A fast matrix majorization-projection method for constrained stress
  minimization in \hbox{MDS}.
\newblock {\em IEEE Transactions on Signal Processing}, 66(3):4331--4346, 2018.

\bibitem{Qi2018}
S.~L. Zhou, N.~H. Xiu, and H.~D. Qi.
\newblock Robust \hbox{E}uclidean embedding via \hbox{EDM} optimization.
\newblock {\em Mathematical Programming Computation}, 2019.

\end{thebibliography}
\end{document}